\documentclass[11pt]{amsart}
\usepackage{amsmath, amsthm, amssymb}
\usepackage{graphicx}
\usepackage[usenames]{color}
\usepackage{srcltx} 
\usepackage{verbatim}
\usepackage{bbm}

\newtheorem{thm}{Theorem}[section]
\newcommand{\bt}{\begin{thm}}
\newcommand{\et}{\end{thm}}

\newtheorem{ex}[thm]{Example}

\newtheorem{cor}[thm]{Corollary}   
\newcommand{\bc}{\begin{cor}}
\newcommand{\ec}{\end{cor}}

\newtheorem{lem}[thm]{Lemma}   
\newcommand{\bl}{\begin{lem}}
\newcommand{\el}{\end{lem}}

\newtheorem{prop}[thm]{Proposition}
\newcommand{\bp}{\begin{prop}}
\newcommand{\ep}{\end{prop}}

\newtheorem{defn}[thm]{Definition}
\newcommand{\bd}{\begin{defn}}    
\newcommand{\ed}{\end{defn}}

\newtheorem{rmrk}[thm]{Remark}   

\newcommand{\br}{\begin{rmrk}}
\newcommand{\er}{\end{rmrk}}

\newcommand{\GHto}{\stackrel { \textrm{GH}}{\longrightarrow} }
\newcommand{\Fto}{\stackrel {\mathcal{F}}{\longrightarrow} }
\newcommand{\VFto}{\stackrel {\mathcal{VF}}{\longrightarrow} }
\newcommand{\VADBto}{\stackrel {VADB}{\longrightarrow} }

\newcommand{\be}{\begin{equation}}

\newcommand{\ee}{\end{equation}}

\newcommand{\C}{\mathbb{C}}

\newcommand{\R}{\mathbb{R}}

\newcommand{\vare}{\varepsilon}

\newcommand{\diam}{\operatorname{Diam}}

\newcommand{\set}{\rm{set}}

\newcommand{\Scal}{{\rm Scal}} 
\newcommand{\disjointunion}{\sqcup}

\newcommand{\mass}{{\mathbf M}}

\newcommand{\area}{\operatorname{Area}}
\newcommand{\vol}{\operatorname{Vol}}

\newcommand{\Sp}{\mathbb{S}}      

\begin{document}

\title[Volume Above Distance Below]{Volume Above Distance Below}

\author{Brian Allen}
\address{Lehman College}
\email{brianallenmath@gmail.com}

\author{Raquel Perales}
\address{CONACYT Researcher for Mexico-Universidad Nacional Autónoma de México}
\email{raquel.peralesaguilar@gmail.com}

\author{Christina Sormani}
\thanks{C. Sormani was partially supported by NSF DMS 1006059.  Some ideas towards this work arose at the IAS Emerging Topics on Scalar Curvature and Convergence that C. Sormani co-organized with M. Gromov in 2018.}
\address{CUNY Graduate Center and Lehman College}
\email{sormanic@gmail.com}





\begin{abstract}
Given a pair of metric tensors $g_1 \ge g_0$ on a Riemannian manifold, $M$, it is well known
that $\vol_1(M) \ge \vol_0(M)$.  Furthermore one has rigidity: the volumes are equal if and only 
if the metric tensors are the same $g_1=g_0$.  Here we prove that
if $g_j \ge g_0$  and $\vol_j(M)\to \vol_0(M)$ then $(M,g_j)$ converge to $(M,g_0)$ in the
volume preserving intrinsic flat sense.  Well known examples demonstrate that one need not
obtain smooth, $C^0$, Lipschitz, or even Gromov-Hausdorff convergence in this setting.  Our theorem
may also be applied as a tool towards proving other open conjectures concerning the geometric stability 
of a variety of rigidity theorems in Riemannian geometry.  
To complete our proof, we provide a novel way of estimating the
intrinsic flat distance between Riemannian manifolds which is interesting in its own right.
\end{abstract}

\maketitle
\vspace{-.1in}
\section{Introduction}\label{sect:intro}


Over the past few decades a number of geometric stability theorems have been proven where one assumes a lower bound on Ricci curvature and proves the Riemannian manifolds are close in the Gromov-Hausdorff sense.  However, without a lower bound on Ricci curvature, there are usually counter examples showing that Gromov-Hausdorff stability is too strong a notion.   Ilmanen's Example depicted in Figure~\ref{fig-Ilmanen} is a sequence of spheres $({\mathbb{S}}^m,g_j)$ with Riemannian metric tensors, $g_j\ge g_0$, that have positive scalar curvature, $\diam(M,g_j) \le D_0$, where $\vol_j({\mathbb{S}}^m) \to \vol_0({\mathbb{S}}^m)$ but the sequence has no converging subsequence.   See 
\cite{Sormani-scalar} for a survey of open stability conjectures and similar counter examples to stability in the Gromov-Hausdorff sense.  

\begin{figure}[h] 
   \center{\includegraphics[width=.6\textwidth]{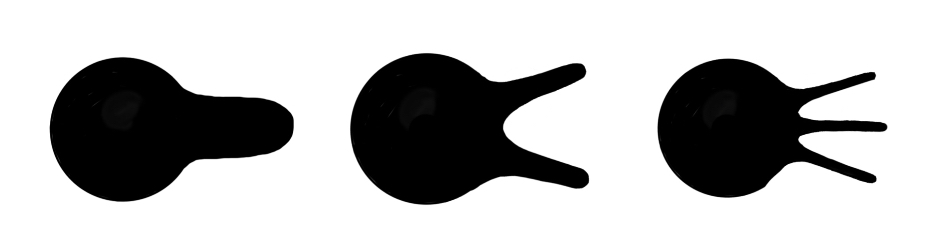}}
\caption{A sequence of spheres $({\mathbb{S}}^m,g_j)$ with $g_j\ge g_0$ and $\vol_j({\mathbb{S}}^m) \to \vol_0({\mathbb{S}}^m)$ that have no Gromov-Hausdorff limit.}
   \label{fig-Ilmanen}
\end{figure}

Gromov has suggested in \cite{Gromov-Dirac} that intrinsic flat convergence might be the right notion of convergence to study for sequences of manifolds with lower bounds on scalar curvature.   Intrinsic flat convergence was first defined by the third author with Wenger in \cite{SW-JDG} building upon the work of Ambrosio-Kirchheim \cite{AK}.  A sequence of oriented manifolds $M_j$ converges in the intrinsic flat sense to $M_0$, $M_j \Fto M_0$ iff they can be embedded by distance preserving maps $\phi_j: M_j \to Z$ into a common complete metric space $Z$ so that the submanifolds $\varphi_j(M_j)$  converge in the flat or weak sense as currents in $Z$ \cite{SW-JDG}.  See Section~\ref{sect:background} for the precise definition.   The sequence is said to converge in the volume preserving intrinsic flat sense $M_j\VFto M_0$ if and only if
\be
M_j \Fto M_0 \textrm{ and } \vol_j(M_j) \to \vol_0(M_0).
\ee  
The third author, Portegies, Lee, and Jauregui have proven many consequences of intrinsic flat convergence in
\cite{Sormani-ArzAsc} \cite{Portegies-Sormani1} \cite{Jauregui-Lee-SWIF}.  In particular balls and spheres within the manifolds also converge in the volume preserving intrinsic flat sense and their filling volumes converge.  In addition Portegies has shown volume preserving intrinsic flat convergence implies measure convergence and that the Laplace spectra of the manifolds semiconverge \cite{Portegies-F-evalue}.

In this paper we prove a new theorem estimating the intrinsic flat distance between two manifolds.    The proof involves a new construction of a common metric space, $Z$, into which we embed the Riemannian manifolds $M_j$ and $M_0$ assuming their distance functions satisfy, $d_j\ge d_0$, and are almost equal on a set of large measure.  We then apply this new estimate to prove the following stability theorem:

\begin{thm}\label{vol-thm-better} 
Suppose we have a fixed  compact oriented Riemannian manifold, $M_0=(M^m,g_0)$,
without boundary and a sequence of  distance non-increasing $C^1$ diffeomorphisms
\be\label{eq-thmbetterF}
F_j: M_j=(M,g_j) \to M_0=(M,g_0)
\ee
i.e.
\be\label{eq-thmbetterd}
d_j(p,q) \ge d_0(F_j(p), F_j(q)) \qquad \forall p,q\in M_0 
\ee
and a uniform upper bound on diameter
\be
\diam_j(M_j) \le D_0
\ee
and volume convergence
\be
\vol_j(M_j) \to \vol_0(M_0)
\ee
then $M_j$ converge to $M_0$ in the volume preserving intrinsic flat sense:
\be
M_j \VFto M_0.
\ee
\end{thm}

Note that this theorem can be seen as a stability result, since it is known that  $g_1\ge g_0 \implies \vol_1\ge \vol_0$ and if the volumes are equal, $\vol_1=\vol_2$, then $g_1=g_2$.  
Our results have applications to important stability theorems as well.  One of the most famous rigidity theorems involving scalar curvature is the Scalar Torus Rigidity Theorem of Schoen-Yau and Gromov-Lawson, which states that if a manifold is homeomorphic to a torus and has $\Scal\ge 0$ then it is isometric to a flat torus
\cite{Schoen-Yau-min-surf}\cite{Gromov-Lawson-torus}.  Gromov has conjectured that this theorem is stable with respect to intrinsic flat convergence \cite{Gromov-Dirac} (cf.~\cite{Sormani-scalar}).  The first author has proven this stability in the warped product setting in joint work with Hernandez-Vazquez, Parise, Payne, and Wang \cite{AHMPPW1}.  The second author has proven this stability in the graph setting in joint work with Cabrera Pacheco
and Ketterer \cite{CPKP19}.  In both these settings the distances are bounded from below and the volumes from above, and thus one may apply Theorem~\ref{vol-thm} as an endplay for their proofs.  The first author has recently applied this paper to prove the stability in the conformal setting \cite{Allen-Conformal-Torus}.

Another important rigidity theorems involving scalar curvature is the Schoen-Yau Positive Mass Theorem \cite{Schoen-Yau-positive-mass}.  This theorem is also conjectured to be stable with respect to intrinsic flat convergence (cf.~\cite{Sormani-scalar}).   As this theorem involves noncompact manifolds, one proves stability by proving intrinsic flat convergence of balls within these spaces.  The first and second author have recently applied the theorems and techniques in this paper to manifolds with boundary  and applied their results to prove the almost rigidity of the positive mass theorem in the graph setting without black holes \cite{Allen-Perales-1}.   In joint work with Huang and Lee, the second author has applied this work to prove it in the graph setting with black holes as well \cite{Huang-Lee-Perales} providing a completely new proof of the results claimed in earlier work of Huang, Lee, and the third author \cite{HLS}.   The second author has also applied these theorems in an upcoming paper to prove the stability of the hyperbolic positive mass theorem in the graph setting in joint work with Cabrera Pacheco \cite{CPP21}.

It is important to note that the hypotheses \eqref{eq-thmbetterF} and \eqref{eq-thmbetterd} of Theorem~\ref{vol-thm-better} are equivalent to assuming
$g_j\ge g_0$ on a fixed manifold (see Theorem~\ref{vol-thm} within). 
It is also important to note that Theorem~\ref{vol-thm-better}  (equivalently Theorem~\ref{vol-thm}) only applies for distances bounded below and volumes bounded above and not visa versa. In Example~\ref{Cinched-Sphere} we see that with $g_j\le g_0$ and $\vol_j \to \vol_0$ the
$M_j$ can fail to converge to $M_0$.   This surprising example of conformal metrics on a sphere first appeared in work of the first and third authors \cite{Allen-Sormani-2} and a similar example with warped product metrics appeared in an earlier paper of theirs \cite{Allen-Sormani}.   These examples
converge to a cinched sphere, a cinched cylinder, or a cinched torus.   In Example~\ref{to-Finsler} we see warped product metrics $g_j$ on a torus 
${\mathbb{T}}^2$ 
such that $g_j \le g_0$ and $\vol_j\to \vol_0$ and yet the
Gromov-Hausdorff and intrinsic flat limit of $({\mathbb{T}}^2, g_j)$ is a Finsler manifold with a symmetric norm that is not an inner product \cite{Allen-Sormani}.  We review these examples in Section~\ref{sect:background}.  Any weaker geometric notion of convergence must also have the same limit, so one can never prove stability for distances above and volumes below.

Some might say that the hypothesis requiring pointwise control on the metric tensors from below
is too strong a hypotheses to be useful in more general settings.  In Corollary~\ref{cor-vol-thm} we see that we only need $C^0$ convergence of the metric tensors from below instead of $g_j\ge g_0$.    In Corollary~\ref{cor-Lp-thm}  we see that $L^p$ convergence with $p\ge m$ can replace the volume convergence.  In Remark~\ref{rmrk-diffeo}
we point out that one really only needs a sequence of diffeomorphic Riemannian manifolds for which one can find a sequence of diffeomorphisms for which the pull backs of the metric tensors satisfy the hypotheses of our theorem or corollary to obtain the conclusion since intrinsic flat convergence is invariant under isometry. 

The main results of this paper suggest a new type of convergence which one could study which we now define.

\begin{defn}
We say that a sequence of Riemannian manifolds $M_j$ without boundary has volume above distance below (VADB) convergence to a Riemannian manifold $M_{\infty}$, $M_j \VADBto M_{\infty}$, if 
\begin{align}
\vol(M_j) &\rightarrow \vol(M_{\infty}), \\\diam(M_j) &\le D_0,
\end{align}
 and there exist $C^1$ diffeomorphisma $\Psi_j: M_{\infty} \rightarrow M_j$ with \begin{align}
 d_{g_j}(\Psi_j(p),\Psi_j(q))\ge d_{g_{\infty}}(p,q)\quad \forall (p,q) \in M_{\infty}\times M_{\infty}.
 \end{align}
\end{defn}

With this new definition our main theorem can be stated $M_j \VADBto M_{\infty}$ implies $M_j \VFto M_{\infty}$. We note that $M_j \VADBto M_{\infty}$ when $M_j=(M,g_j)$, $M_{\infty}=(M,g_{\infty})$, $\vol(M_j) \rightarrow \vol(M_{\infty})$, $\diam(M_j)\le D_0$ and $g_j \ge g_{\infty}$. We can also rescale $M_j$ to study manifolds so that $g_j \ge (1 - \frac{1}{j}) g_{\infty}$. In \cite{ScalarSurvey-Sormani} the third named author proposes a few natural conjectures involving this new notion of convergence.

In Section~\ref{sect:background}, we briefly provide sufficient background on integral current spaces and the intrinsic flat distance so as to make this paper understandable to those who are new to this notion.  We refer the reader also to \cite{Sormani-scalar} for a longer review.   We also review key examples of the first and third authors which are relevant to this paper as well as their earlier versions of Theorem~\ref{vol-thm} which imply Gromov-Hausdorff as well as intrinsic flat convergence under significantly stronger hypotheses.   

In Section~\ref{sect:NewSWIFEstimate}, we prove Theorem~\ref{est-SWIF} which
provides the new method of estimating the intrinsic flat distance between two Riemannian manifolds. The proof involves a new construction of a common metric space, $Z$, into which we embed the Riemannian manifolds $M_j$ and $M_0$ such that $d_j\ge d_0$ and $d_j$ is close to $d_0$ on a good set of almost full measure.  

In Section~\ref{sect:GoodSet}, we show how to construct a good set with almost full measure where we can guarantee control on the distance function on $M_j$. A key insight is to use Egoroff's Theorem in order to go from pointwise convergence of distance almost everywhere to uniform convergence on a subset of $M\times M$ of almost full measure. The bulk of the section is then devoted to describing a good subset of $M$ of almost full measure which satisfies the necessary hypotheses of Section \ref{sect:NewSWIFEstimate} in order to estimate the Intrinsic Flat distance.

In Section~\ref{sect:ProofMainThm}, we put all of these results together in order to prove Theorem \ref{vol-thm}.  We also state and prove Corollary \ref{cor-vol-thm}.   The paper closes with a section of open problems.

We would like to thank Misha Gromov for his interest in intrinsic flat convergence and all the attendees of the IAS Emerging Topics on Scalar Curvature and Convergence.   We would particularly like to thank Ian Adelstein, Lucas Ambrozio, Armando Cabrera Pacheco, Alessandro Carlotto, Michael Eichmair, Lan-Hsuan Huang, Jeff Jauregui, Demetre Kazaras, Christian Ketterer, Sajjad Lakzian, Dan Lee, Chao Li, Yevgeny Liokumovich, Siyuan Lu, Fernando Coda Marques, Elena Maeder-Baumdicker, Andrea Malchiodi, Yashar Memarian, Pengzi Miao, Frank Morgan, Alexander Nabutovsky, Andre Neves, Alec Payne, Jacobus Portegies, Regina Rotman, Richard Schoen, Craig Sutton, Shengwen Wang, Guofang Wei, Franco Vargas Pallete, Robert Young, Ruobing Zhang, and Xin Zhou for intriguing discussions with us related to intrinsic flat convergence and stability at this event and at other workshops at Yale and NYU.    We would also like to thank the many participants in the 2020 Virtual Workshop on Ricci and Scalar Curvature in honor of Gromov for their many thoughts and suggestions of further applications of this work.


\section{Background}\label{sect:background}
The main theorem in this paper is a stability or almost rigidity theorem.  In this section we first
write a restatement of the main theorem with different equivalent hypotheses and prove the equivalence
using basic Riemannian geometry.  We then review the corresponding rigidity theorem and provide a new proof of that
theorem which gives in some sense an outline of our proof of the rigidity theorem.  Next we review the key
aspects of intrinsic flat convergence and work of Ambrosio-Kirchheim
needed to understand the statement of our main theorem and its proof.
We then review an older theorem of Huang-Lee and the third author which is used to prove stronger convergence
and apply this older theorem to present the examples mentioned in the introduction.  The final subsection reviews a key theorem by the first and third authors which will be applied to prove the main theorem.   

\subsection{Restating the Main Theorem}

Before we begin we would like to clarify that our main theorem is equivalent to the following theorem:

\begin{thm}\label{vol-thm} 
Suppose we have a fixed  compact oriented Riemannian manifold, $M_0=(M^m,g_0)$,
without boundary and
a sequence of metric tensors $g_j$ on $M$ defining $M_j=(M, g_j)$ with
\be \label{g_j-below-vol-thm}
g_0(v,v) \le g_j(v,v) \qquad \forall v\in TM
\ee 
and a uniform upper bound on diameter
\be
\diam_j(M_j) \le D_0
\ee
and volume convergence
\be
\vol_j(M_j) \to \vol_0(M_0)
\ee
then $M_j$ converge to $M_0$ in the volume preserving intrinsic flat sense:
\be
M_j \VFto M_0.
\ee
\end{thm}

The equivalence can be seen by pushing forward the metric $g_j$ to $M_0$ using the  map $F_j: M_j \to M_0$
and applying the two lemmas below. 

\begin{lem}\label{DistToMetric}
Let $M_1=(M,g_1)$ and $M_0=(M,g_0)$ be Riemannian manifolds and  $F: M_1 \rightarrow M_0$ be a $C^1$ diffeomorphism then 
\begin{equation}
g_0 (dF(v), dF(v))  \leq g_1 (v,v) \qquad \forall v\in TM_1.
\end{equation}
iff
\be
d_0(F(p),F(q))\le d_1(p,q)\qquad \forall p,q\in M_1
\ee
\end{lem}

\begin{proof}
First recall by the definition of the Riemannian distance
\be
d_g(p,q) = \inf \{ L_{g}(C): \, C(0)=p,\, C(1)=q\} \textrm{ where } L_g(C)=\int_0^1 g(C',C')^{1/2} \, dt.
\ee
Thus it is easy to see that
\be
d_0(F(p),F(q))\le L_{g_0}(F\circ C) \le L_{g_1}(C)
\ee
and taking the infimum we have $d_0(F(p),F(q))\le d_1(p,q)$.

On the other hand, if we let $C:(-1,1)\to M_1$ be any smooth curve
such that $C(0)=p$ and $C'(0)=v$.  Then we can calculate,
\begin{align}
g_1 (v,v)   &= \lim_{t \to 0}\frac{d_1(C(t), p)^2}{t^2}
\\& \ge \lim_{t \to 0}\frac{d_0(F(C(t)), F(p))^2}{t^2}\label{DistNonIncreasing}
\\& = g_0(dF(v), dF(v)),
\end{align}
where we are using the distance non-increasing assumption in \eqref{DistNonIncreasing}. 
\end{proof}

\subsection{Volume-Distance Rigidity Theorem}

Our main theorem is an almost rigidity theorem for the following well known rigidity theorem.  

\begin{thm}\label{rigidity-thm}
Suppose $M_1=(M,g_1)$ and $M_0=(M,g_0)$ are a pair of Riemannian manifolds, and $F: M_1 \to M_0$
is a $C^1$ diffeomorphism that is distance non-increasing 
\be
d_0(F(p),F(q))\le d_1(p,q)\qquad \forall p,q\in M_1
\ee
then
\be \label{v-0-v-1}
\vol_0(M_0) \le \vol_1(M_1).
\ee
Furthermore if $\vol_1(M_1)=\vol_0(M_0)$ then they are isometric
\be \label{d-0-d-1}
d_0(F(p),F(q))= d_1(p,q)\qquad \forall p,q\in M_1.
\ee
\end{thm}

For completeness of exposition we include a proof of this rigidity theorem and then follow this with an explanation
as to the difficulties which arise when trying to prove an almost rigidity version of this theorem.

\begin{proof}
We begin by proving the inequality (\ref{v-0-v-1}) through a series of inequalities. Starting with
\be
d_0(F(p),F(q))\le d_1(p,q)\qquad \forall p,q\in M_1
\ee
and applying Lemma~\ref{DistToMetric}, we have
\begin{equation}
g_0 (dF(v), dF(v))  \leq g_1 (v,v) \qquad \forall v\in TM_1.
\end{equation}
By pushing the metric $g_0$ forward through the map $F$ we can 
 without loss of generality consider $g_1 = F^*g_1$ in order to write
\begin{equation}
g_1 (v, v)  \ge g_0 (v,v) \qquad \forall v\in TM_0.
\end{equation}
In particular the eigenvalues of $g_1$ with respect to $g_0$:
\be
\lambda \textrm{ such that } \exists v_\lambda \textrm{ such that } g_1 (v_\lambda, v_\lambda) =\lambda g_0 (v_\lambda,v_\lambda)
\ee
must all have $\lambda\ge 1$.   Taking the product of these eigenvalues we have
\be
Det_{g_0}(g_1) \ge 1
\ee
Since for any Borel set $A \subset M$
\be
\vol_{g_1}(A) =\int_A \sqrt{Det_{g_0}(g_1)} \, dvol{g_0} \ge \int_A 1 \, dvol_{g_0} = \vol_{g_0}(A)
\ee
we have (\ref{v-0-v-1}) as desired.  

Now we prove the rigidity by observing that all the inequalities above become equalities
when the final line has an equality.  We start with $\vol_1(M_1)=\vol_0(M_0)$, then we are forced to have equality for any
Borel set $A \subset M$
\be
\vol_{g_1}(A) =\int_A \sqrt{Det_{g_0}(g_1)} \, dvol_{g_0} = \int_A 1 \, dvol_{g_0} = \vol_{g_0}(A)
\ee
and so by continuity
\be
Det_{g_0}(g_1) = 1.
\ee
Hence all the eigenvalues are equal to $1$  and hence
\be
g_1=g_0.
\ee
Returning to the use of $F$ we have 
\be
g_0 (dF(v), dF(v)) = g_1 (v,v)
\ee
which by Lemma~\ref{DistToMetric} gives us (\ref{d-0-d-1}).
\end{proof}

To prove an almost rigidity theorem one then starts with an almost equality in the final line,
or assume 
\be
\lim_{j\to \infty} \vol_j(M_j)=\vol_0(M_0).
\ee
We can then show that for any Borel set $A\subset M_0$
\be
\vol_{g_j}(A) =\int_A \sqrt{Det_{g_0}(g_j)} \, dvol_{g_0} \to  \int_A 1 \, dvol_{g_0} = \vol_{g_0}(A)
\ee
which will be done within this paper carefully.
However one cannot conclude
\be
\sqrt{Det_{g_0}(g_j)} \to 1.
\ee 
In fact we will see this is not well controlled at all.  Instead we will apply a theorem of the first
and third authors from \cite{Allen-Sormani-2} which chooses special sets $\mathcal{T}=\mathcal{T}_{p,q}$ that can be thought of as thin cylinders 
around geodesics from $p$ to $q$ so that
\be
\vol_{g_j}( \mathcal T_{p,q}) \textrm{ is close to } \omega_{m-1}\epsilon^{m-1} d_j(p,q)
\ee
and eventually show that there is a subsequence such that
\be 
d_j (p,q) \to d_0(p,q) \textrm{ pointwise almost everywhere } (p,q) \in M\times M.
\ee
We review this theorem in the final subsection of the background.

This paper is dedicated to proving intrinsic flat convergence using this
control on the distances combined with the bounds on volume and diameter.  

\subsection{Review of the Intrinsic Flat Distance}\label{subsect:SWIFBackground}

In \cite{SW-JDG}, Sormani-Wenger  defined the intrinsic flat distance between pairs of oriented Riemannian manifolds with boundary as follows:
\be\label{defn-IF}
d_{\mathcal{F}}(M_1^m, M_2^m) = \inf d_F^Z\left(\varphi_{1\#} [[M_1]], \varphi_{2\#} [[M_2]]\right)
\ee
where the infimum is taken over all complete metric spaces $Z$ and all distance preserving maps $\varphi_i: M_i \to Z$,
\be
d_Z(\varphi_i(p), \varphi_i(q)) = d_i(p,q) \qquad \forall p,q\in M_i.
\ee
Here the flat distance between the images of $M_i$, viewed as integral currents,
\be
T_i=\varphi_{i\#} [[M_i]]\in I_m(Z),
\ee
 is defined
\begin{equation} \label{dFZ}
d_F^Z(T_1, T_2)= \inf \left(\mass(A) + \mass(B)\right)
\end{equation}
where the infimum is over all integral currents, $ A \in I_m(Z), \, B \in I_{m+1}(Z)$ such that
$A + \partial B  =  T_1 - T_2$.
 To rigorously understand this definition one needs Ambrosio-Kirchheim theory which we review the essential 
 elements of below.  

The intuitive idea is that the intrinsic flat distance is measuring the volume between the two Riemannian manifolds.  To estimate the intrinsic flat distance, one first embeds them into a common metric space $Z$ without distorting distances,
then one finds an oriented rectifiable submanifold $A$ so that 
the images $\varphi_i(M_i)$ and $A$ form the boundary of an oriented rectifiable submanifold $B$ of one dimension higher,  and then one bounds the intrinsic flat distance from above by the sum of the volumes of $A$ and $B$.  One needs generalized weighted submanifolds called integral currents to find the precise value of the intrinsic flat distance.  These
currents were first defined by Federer-Flemming \cite{FF} in Euclidean space
and
by Ambrosio-Kirchheim for complete metric spaces in \cite{AK}.

In \cite{AK}, Ambrosio-Kirchheim defined the class of $m$-dimensional integral currents, $T\in I_m(Z)$,
in a complete metric space $Z$ as integer rectifiable currents whose boundaries are also integer rectifiable.  
Since there are
no differential forms on metric spaces, Ambrosio-Kirchheim defined currents as acting on tuples
$(f, \pi_1,...,\pi_m)$, where $f: Z\to {\mathbb{R}}$ is a bounded Lipschitz function and each
$\pi_j: Z\to {\mathbb{R}}$ is Lipschitz, rather than forms $f d\pi_1 \wedge \cdots d\pi_m$.  They define
\be
\varphi_{\#} [[M]](f, \pi_1,...\pi_m) = \int_{M} (f\circ \varphi)\, d(\pi_1 \circ \varphi)\wedge \cdots \wedge d(\pi_m \circ \varphi)
\ee
which is well defined for any oriented Riemannian manifold with or without boundary and any Lipschitz function $\varphi: M\to Z$.
More generally an $m$ dimensional integer rectifiable current, $T$, can be parametrized by
a countable collection of biLipschitz charts, $\varphi_i: A_i \to \varphi_i(A_i)\subset Z$ where 
$A_i$ are Borel in ${\mathbb R}^m$ with pairwise disjoint images and integer weights $\theta_i\in {\mathbb Z}$ such that
\be
T(f, \pi_1,...\pi_m) = \sum_{i=1}^\infty \theta_i \int_{A_i} (f\circ \varphi_i)\, d(\pi_1 \circ \varphi_i)\wedge \cdots \wedge d(\pi_m \circ \varphi_i)
\ee
has finite mass, $\mass(T)=||T||(Z)$.  Their definition of mass and mass measure $||T||$
is subtle for currents in general but in Section 9 of \cite{AK}, they prove that for rectifiable currents
\be
||T||=  \lambda \theta \mathcal{H}^m
\ee
where $\theta$ is an integer valued function and the area factor
$\lambda:  \set(T)  \to \R$ is a measurable function bounded above by 
\be \label{C_m}
C_m=2^m/\omega_m \textrm{ where } \omega_m=\vol_{{\mathbb E}^m}(B_0(1)).
\ee
So that
\be
\mass(T) \le C_m  \sum_{i=1}^\infty |\theta_i| \mathcal{H}^m( \varphi_i(A_i)) < \infty.
\ee

Ambrosio-Kirchheim define the boundary of any current  to be
\be
\partial T(f, \pi_1,...\pi_{m-1})= T(1,f, \pi_1,...\pi_{m-1}).
\ee
This agrees with the notion of the boundary of a submanifold:
\begin{eqnarray}
\qquad 
\partial \varphi_{\#}[[M]](f, \pi_1,...\pi_{m-1})&=&
\varphi_{\#}[[M]](1, f, \pi_1,...\pi_{m-1})\\
&=&
\int_{M} d(f\circ \varphi)\wedge d(\pi_1 \circ \varphi)\wedge \cdots \wedge d(\pi_m \circ \varphi)\\
&=& \int_{\partial M} (f\circ \varphi)\, d(\pi_1 \circ \varphi)\wedge \cdots \wedge d(\pi_m \circ \varphi)\\
&=&\varphi_{\#}[[\partial M]](f, \pi_1,...\pi_{m-1}).
\end{eqnarray}
They define an $m$ dimensional integral current to be an integer rectifiable current, $T$, whose
boundary $\partial T$ is also integer rectifiable.  With this information(\ref{dFZ}) is well defined and finite.

Note that the definition of intrinsic flat convergence in (\ref{defn-IF}) does not require $M_j$ to be smooth Riemannian
manifolds.  In \cite{SW-JDG} the distance is defined between a larger class of spaces called integral current spaces.  
We do not need to consider general integral current spaces in this paper.  However it is worth observing that the
definition as in (\ref{defn-IF}) can be understood for a pair of $C^1$ oriented manifolds, $M_j$, endowed with
metric tensors $g_j$ that need not even be continuous, just so long as the $C^0$ charts are biLipschitz with 
respect to the distance functions:
\be
d_j: M_j \times M_j \to [0,\infty)
\ee
defined by
\be \label{djdefn}
d_j(p,q) = \inf\{L_j(C):\, C(0)=p, \, C(1)=q\}
\ee
where
\be\label{Ljdefn}
L_j(C)=\int_0^1 g_j(C'(s),C'(s))^{1/2}\, ds.
\ee
We see that such manifolds can arise as intrinsic flat limits of sequences of smooth
Riemannian manifolds in the next section.

\subsection{Convergence of Metrics on a Fixed Manifold}

In the Appendix to \cite{HLS}, Lan-Hsuan Huang, Dan Lee, and the third author considered sequences
of distance functions on a fixed metric space just as we do here except with significantly stronger hypotheses.  
Here we restate their appendix theorem in the simplified setting where $M^m$
is a manifold and $d_j$ are defined as in (\ref{djdefn})-(\ref{Ljdefn}).   We state this theorem because it is applied to prove the convergence of some of the examples and because its proof inspired some of our ideas.  We do not apply this theorem to prove our Theorem~\ref{vol-thm} because the hypotheses of this theorem are too strong.

\begin{thm}\label{app-thm}\cite{HLS}
Given $(M, d_0)$ Riemannian without boundary and fix
$\lambda>0$, suppose that
$d_j$ are length metrics on $M$ such that
\be\label{HLS-d_j}
\lambda \ge \frac{d_j(p,q)}{d_0(p,q)} \ge \frac{1}{\lambda}.
\ee
Then there exists a subsequence, also denoted $d_j$,
and a length metric $d_\infty$ satisfying (\ref{HLS-d_j}) such that
$d_j$ converges uniformly to $d_\infty$:
\be\label{HLS-epsj}
\varepsilon_j= \sup\left\{|d_j(p,q)-d_\infty(p,q)|:\,\, p,q\in X\right\} \to 0.
\ee 
and $M_j$ converges in the intrinsic flat and Gromov-Hausdorff sense to $M_\infty$:
\be
M_j \Fto M_\infty \textrm{ and } M_j \GHto M_\infty
\ee
where $M_j=(M,d_j)$ and $M_\infty=(M, d_\infty)$.
\end{thm}

Note that the hypotheses of our main theorem, Theorem~\ref{vol-thm}, do not imply the upper bound in the 
hypothesis (\ref{HLS-d_j}) of Theorem~\ref{app-thm}.   Yet this upper bound is crucially applied in the Appendix to \cite{HLS} to obtain the
existence of a subsequence which converges uniformly as in (\ref{HLS-epsj}) and that uniform convergence
is applied to provide an 
 explicit construction of the common metric space 
\be
Z_j= [-\varepsilon_j, \varepsilon_j] \times M
\ee
with an explicit distance function $d_j'$ on $Z_j$ such that
\be \label{iso-}
d'_j((-\varepsilon_j,p), (-\varepsilon_j,q)) = d_j(p,q)
\ee
\be \label{iso+}
d'_j((\varepsilon_j,p), (\varepsilon_j,q)) = d_\infty(p,q).
\ee
Taking $A=0$ and $B=[[Z_j]]$ the intrinsic flat distance is then 
proven to be
\be\label{Fj}
d_{\mathcal{F}}\left(M_j, M_\infty \right) \le \mass(A)+\mass (B) \le 
2^{(n+1)/2} \lambda^{n+1} 2\varepsilon_j \vol_0(M) \to 0.
\ee
The proof of our main theorem, Theorem~\ref{vol-thm} will also involve the explicit construction of a space $Z$.

\subsection{Examples Without Distance Bounded Below}\label{subsect:PreviousWorkContrasting}

In \cite{Allen-Sormani} and \cite{Allen-Sormani-2}, the first and third authors presented a number of examples comparing and contrasting various notions of convergence for Riemannian manifolds.   The examples in \cite{Allen-Sormani} were warped products and the examples in \cite{Allen-Sormani-2} were conformal.   Here we present two crucial examples from these papers demonstrating the importance of the lower bounds on distance, $g_j \ge g_0$ in Theorem \ref{vol-thm}
and the $C^0$ control on the metric tensor from below in Corollary~\ref{cor-vol-thm}.  In these examples
we have an upper bound on distance $g_j \le g_0$ and $\vol_j \to \vol_0$ but $M_j$
converge to something other than $M_0$. 

In the first example, which is Example 3.1 in \cite{Allen-Sormani-2}, we have a sequence of conformal metric tensors on the sphere that are shrunk near the equator so that one obtains a cinched sphere as the  intrinsic flat limit instead of the round sphere.   See also
Example 3.4 in \cite{Allen-Sormani}.

\begin{ex} \label{Cinched-Sphere}  \cite{Allen-Sormani-2}
Let $g_0$ be the standard round metric on the sphere, ${\mathbb S}^m$.   Let $g_j=f_j^2 g_0$ be
metrics conformal to $g_0$ with smooth conformal factors, $f_j$,
that are radially defined from the north pole with a cinch at the equator as follows:  
 \be
 f_j(r)=
 \begin{cases}
 1 & r\in[0,\pi/2- 1/j]
 \\  h(jr-\pi/2) & r\in[\pi/2- 1/j, \pi/2+ 1/j]
 \\ 1 &r\in [\pi/2+ 1/j, \pi]
 \end{cases}
\ee
where $h:[-1,1]\rightarrow \R$ is an even function 
decreasing to $h(0)=h_0\in (0,1)$ and then
increasing back up to $h(1)=1$.   Observe that
\begin{align}
g_0 \not \le g_j \textrm{ but instead } g_j \le g_0,
\end{align}
\be
\vol_j(\Sp^m) \to \vol_0(\Sp^m),
\ee
and 
\be
\diam_j(\Sp^m) \le \diam_0(\Sp^m).
\ee
In \cite{Allen-Sormani-2}, it is proven that
\be
M_j \Fto M_{\infty}
\ee
where  $M_{\infty}$ is not isometric to $\Sp^m$. Instead $M_{\infty}$ is endowed with the conformal metric,
$g_\infty=f_\infty^2g_0$ with 
a piecewise conformal factor that is not continuous:
 \be
 f_{\infty}(r)=
 \begin{cases}
 h_0 & r=\pi/2
 \\  1 &\text{ otherwise}
 \end{cases}.
 \ee
 The distance, $d_\infty$, between pairs of points near the equator in this limit space is defined as
 in (\ref{djdefn})-(\ref{Ljdefn}).  It 
 is achieved by geodesics which run to the
 equator, and then around inside the cinched equator, and then out again.  
 
To prove this convergence in \cite{Allen-Sormani-2}, Theorem~\ref{app-thm} was applied to show a subsequence $d_j$ converges uniformly to some distance function and then it was shown explicitly that $d_j$ converge pointwise to $d_\infty$, thus $d_\infty$ is the
uniform limit of in fact any subsequence.  Theorem~\ref{app-thm} then implied $(M,d_\infty)$ was the
intrinsic flat and Gromov-Hausdorff limit as well.
\end{ex}

Some might point out that in the above example the limit space is locally isometric to a standard sphere almost everywhere, and that perhaps it is thus not so different from a standard sphere.  In the next example, which is Example 3.12 in \cite{Allen-Sormani}, we see that the limit space need not even be locally isometric
to $M_0$ anywhere.  In fact it need not even be Riemannian but could instead be a Finsler manifold with a symmetric norm that is not an inner product.

\begin{ex} \label{to-Finsler} \cite{Allen-Sormani}
Let $M={\mathbb T}^2$ be a torus with warped product metrics $g_j=dr^2 + f_j(r)^2 d\theta^2$
where smooth $f_j: [-\pi, \pi]\to [1,5]$ are defined so that
\be
g_j \le g_0=dr^2 + 5^2 d\theta^2 \textrm{ and } \vol_j \to \vol_0
\ee
but the  $f_j$ are cinched on an increasingly dense set, so that the $d_j$ converge
uniformly to a distance, $d_\infty$, that is Finsler, and $M_j\Fto M_\infty$.
Taking
\begin{eqnarray}
 S&=&\left\{s_{i,j}=-\pi + \tfrac{2\pi i}{2^j}\,: \,  i=1,2,... (2^j-1),\, j\in \mathbb{N}\right\}\\
&=& \left\{-\pi + \tfrac{2\pi}{2},-\pi+\tfrac{2\pi}{4}, -\pi + \tfrac{2\pi 2}{4}, -\pi + \tfrac{2\pi3}{4}, 
-\pi + \tfrac{2\pi}{8},...\right\}
 \end{eqnarray}
 which is dense in $[-\pi,\pi]$
 and
 \be
 \{\delta_j=(1/2)^{2j}:\, j \in \mathbb{N}\} =\{1/4, 1/16, 1/32,...\} 
 \ee
 one can define the functions $f_j$  that are cinched on this set $S$ as follows
  \be
 f_j(r)=
 \begin{cases}
  h((r-s_{i,j})/\delta_j ) & r\in [s_{i,j}-\delta_j, s_{i,j} +\delta_j] \textrm{ for } i =1...2^j-1
 \\ 5 & \textrm{ elsewhere }
 \end{cases}
\ee
where $h$ is an even function such that 
$h(-1)=5$  decreasing down to $h(0)=1$ and then
increasing back up to $h(1)=5$.  

Then $f_j(r)\ge 1$ converges pointwise
to 1 on the dense set, $S$, and pointwise to $5$ elsewhere.   This causes the existence of so many shorter
paths in the limit space that $d_j$ are shown in \cite{Allen-Sormani-2} to converge pointwise to
\be
d_{\infty}((s_1,\theta_1),(s_2,\theta_2))= \min \left\{  \sqrt{s^2 + 5^2 \theta^2} ,
 s\left( \tfrac{\sqrt{24}}{5} \right)+ \theta \right\}
\ee
where $s=d_{{\mathbb S}^1}(s_1,s_2)$ and $\theta=d_{{\mathbb S}^1}(\theta_1, \theta_2)$.
To obtain uniform, intrinsic flat and Gromov-Hausdorff convergence, Theorem~\ref{app-thm} was applied.
\end{ex}

We  encourage the reader to explore the other examples of sequences of conformal Riemannian metrics given by Allen and Sormani in \cite{Allen-Sormani-2} which provide further understanding of relationship between important notions of convergence in geometric analysis.   Keep in mind that the examples presented here are particularly nice because
we do have the strong two sided bounds required to apply Theorem~\ref{app-thm}.

\subsection{Ilmanen Example}

We now provide the details of Example~\ref{ex-Ilmanen} that was depicted in 
Figure~\ref{fig-Ilmanen} in the introduction.  It is a sequence of spheres $({\mathbb{S}}^3,g_j)$ with $g_j\ge g_0$ and $\vol_j({\mathbb{S}}^m) \to \vol_0({\mathbb{S}}^3)$ that has no Gromov-Hausdorff limit
but by our new Theorem~\ref{vol-thm} converges in the intrinsic flat sense to the standard round sphere, $({\mathbb{S}}^m, g_0)$.   In dimension $m=3$ this was presented in a talk by Ilmanen as an example of a sequence with positive scalar curvature and no Gromov-Hausdorff limit that ought to converge in some sense to
the standard sphere.  That example was described in more detail in the Appendix of \cite{SW-JDG} by the third author in part to justify that intrinsic flat convergence is the right notion of convergence for such a sequence.  Here we modify the construction slightly so that it is easier to see, and then discuss how it is related to this paper.   It is an important example to keep in mind when reading the proof of Theorem~\ref{vol-thm}.

\begin{ex} \label{ex-Ilmanen}  
Ilmanen presented a sequence of three dimensional spheres $M_j=({\mathbb S}^3, g_j)$
of positive scalar curvature
that had increasingly many increasingly thin wells with no Gromov-Hausdorff limit as in
Figure~\ref{fig-Ilmanen} (cf. \cite{SW-JDG}).
To construct the sequence one starts with $M_0=({\mathbb S}^3, g_0)$ with the standard round metric $g_0$.
One then choses a finite collection of points
\be
Q_j=\{q^j_1,q^j_2,...q^j_j\} \subset M_0
\ee
and a radius $\rho_j \to 0$ so that balls of radius $\rho_j$ centered at $q\in Q_j$ are pairwise disjoint.
They can be chosen to be increasingly dense and in fact we can always take the first and last 
to be opposite poles
\be
q^j_1=q_+ \textrm{ and } q^j_j =q_- \textrm{ so that } d_0(q^j_1, q^j_j)=\pi.
\ee
One then fixes a length $R>0$ and constructs
wells $(W_j,g_j)$ which are rotationally symmetric balls with positive scalar curvature such that
$B_p(R)\subset W_j$  has $\area_j(\partial B_p(r))$ increased from $0$ at $r=0$ to $4\pi \rho_j^2>0$ at
$r=R$.  At $r=R$ each well $W_j$ is smoothly attached to the standard sphere replacing each ball, $B_{q^j_i}(\rho_j) 
\subset M_0$ with a ball  $B_{p^j_i}(R) \subset M_j$.  Outside of the wells $g_j=g_0$.

Ilmanen took $\rho_j \to 0$ fast enough that $j \vol_j(W_j) \le j R 4\pi \rho_j^2 \to 0$ so that
\be
\vol_j(M_j) 
= \vol_0 \left( M_0\setminus \bigcup_{q\in Q_j} B_q(\rho_j) \right) + j \vol_j(W_j) \to  \vol_0(M_0).
\ee
Note also that $\diam_j(M_j) \le \pi + 2R$.   

One can construct a distance non-increasing diffeomorphism $F_j: M_j\to M_0$ by taking $F_j$ to be the identity map away from the wells and setting $F_j$ to be rotationally symmetric on each well determined by the requirement that
\be
F_j:\, \partial B_p(r) \subset M_j \,\,\to \,\, \partial B_q(\rho_j(r))
\textrm{ where } 4\pi (\rho_j(r))^2 = \area_j(\partial B_p(r)).
\ee
So by our new Theorem~\ref{vol-thm} we have a new proof that
\be
M_j=({\mathbb S}^m, d_j) \Fto M_0=({\mathbb S}^m, d_0).
\ee

As mentioned in the first section of the background, we can replace $g_j$ with $F_j^*g_j$ to view them as a sequence of metric tensors on a fixed manifold and the distance functions 
\be
d_j: M \times M \to [0, D]
\ee 
as a sequence of distance functions on a fixed manifold.  Note that we do not have pointwise convergence
of the distance functions.   Take for example the tips of the wells at the poles correspond to the
poles $p_+$ and $p_-$, and that
\be
d_0(p_-, p_+)= \pi \textrm{ however } d_j(p_-, p_+) \to \pi + 2 R
\ee
due to the depth of the wells.  
\end{ex}

We discuss convergence of the distance functions pointwise almost everywhere in the
next section.

\subsection{Volume to Pointwise Almost Everywhere Convergence}\label{subsect:PreviousWorkRelating}

An important theorem established by the first and third authors as Theorem 4.4 in \cite{Allen-Sormani-2} gives a way of obtaining pointwise convergence almost everywhere of the distance functions $d_j$ from the hypotheses of Theorem \ref{vol-thm}. We review the statement and proof of this theorem here since it will be applied in Section \ref{sect:ProofMainThm} to prove Theorem \ref{vol-thm}. 

\begin{thm}\label{PointwiseConvergenceAE}
If $(M, g_j)$ are compact continuous Riemannian manifolds without boundary  and $(M, g_0)$ is a smooth Riemannian manifold such that
\be
g_j(v,v) \ge g_0(v,v) \qquad \forall v\in T_pM
\ee
and
\be
\vol_j(M) \to \vol_0(M)
\ee
then there exists a subsequence such that 
\be
\lim_{j\to \infty} d_j(p,q) = d_0(p,q) \textrm{ pointwise a.e. } (p,q) \in M\times M,
\ee
where we use the volume with respect to $g_0 \oplus g_0$ on $M \times M$.
\end{thm}

Since this theorem is so fundamental to the proof of our results in this paper, we outline the proof here.  The details
required for all the estimates are in \cite{Allen-Sormani-2}.   When reading the proof consult Figure~\ref{fig-APS-tubes}.

\begin{proof}
In \cite{Allen-Sormani-2}, the first and third authors first show 
that for any Borel set $\vol_j(U) \ge \vol_0(U)$ because $g_j\ge g_0$.
Since $\vol_j(M)\to \vol_0(M)$, they further prove that
\be\label{VolumeSetsConverge}
    Vol_j(U) = \int_U \sqrt{Det_{g_0}(g_j)}\, dV_{g_0} \rightarrow Vol_0(U)=\int_U dV_{g_0}.
\ee
They next apply this to tubes of $g_0$-geodesics as depicted in Figure~\ref{fig-APS-tubes}.

\begin{figure}[h] 
   \center{\includegraphics[width=.6\textwidth]{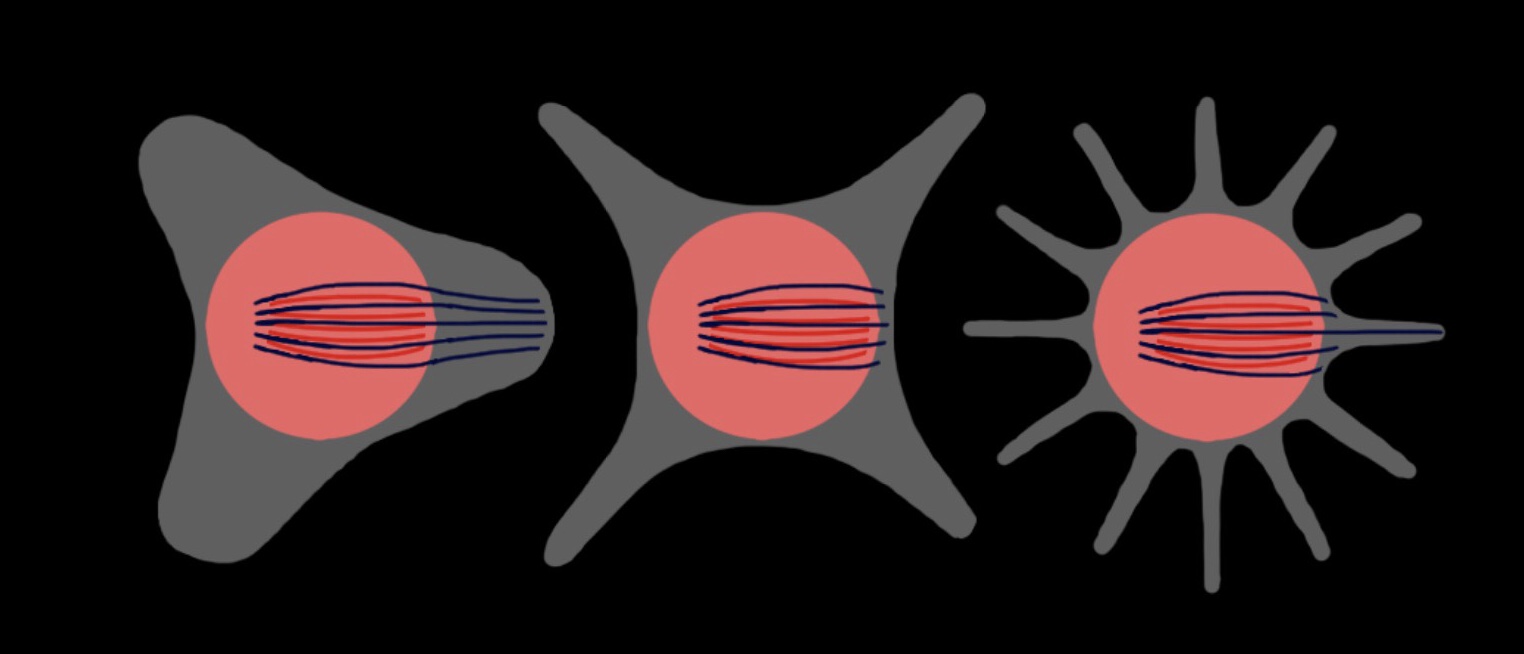}}
\caption{A tube $\mathcal{T}$ foliated by $g_0$-geodesics, $\gamma$, 
with $L_j(\gamma)\ge L_0(\gamma)$ has $\vol_j(\mathcal{T})\to \vol_0(\mathcal{T})$ so $L_j(\gamma)\to L_0(\gamma)$ for almost every $\gamma$ but not for $\gamma$ ending at a tip.}
   \label{fig-APS-tubes}
\end{figure}

Since they only wish to show pointwise almost everywhere convergence, they consider
 $p,q \in M$ so that $q$ is not a cut point of $p$ with respect to $g_0$:
 \be
 \mathcal{U}=\{(p,q):\, q \notin CutLocus_{g_0}(p)\} \subset M \times M.
 \ee
   They choose 
 \be
 v=v_{p,q}\in T_pM \textrm{ such that } \exp_p(v_{p,q}) = q
 \ee
 and $\gamma_p(t)=\exp_p(tv)$ is $g_0$-length minimizing from $p$ to $q$: $L_0(\gamma_p)=d_0(p,q)$.
The goal is to define a parametrized tube around the geodesic from $p$ to $q$. In order to accomplish this they define
\begin{align}
    w \in N_{v,\alpha,p}=\{w: w \in S_k\subset T_pM , |w|_{g_0} < \alpha\}
\end{align}
where $S_k$ is a sphere of radius $\frac{1}{k}$ (or a hyperplane if $k=0$) through the origin in $T_pM$ which is carefully chosen to avoid
focal points in the foliation constructed below and $\alpha>0$ is chosen small enough so that
they can extend $v$ uniquely to $T_{p'}M$ for every $p'=\exp_p(w)$, $w \in N_{v,\alpha,p}$ by choosing a $v \in T_{p'}M$ so that $v \perp \exp_p(N_{v,\alpha,p})$, of the same length as $v \in T_pM$, and $v$ is a continuous vector field on $ \exp_p(N_{v,\alpha,p})$. Then the foliation is defined: 
 \begin{align}
    \mathcal{T}_{v, \alpha,p} = \{\gamma_{p'}(t):p' = \exp_p(w), w \in N_{v,\alpha,p}, 0 \le t \le 1\},
\end{align}
created using a foliation by length minimizing $g_0$-geodesics
\begin{align}
    \gamma_{p'}(t) = \exp_{p'}(tv), \quad 0 \le t \le 1 \textrm{ running from } p' \textrm{ to } q'
\end{align}
where $p' = \exp_p(w)$.

  Keep in mind that
\be\label{71}
L_j(\gamma_{p'})\ge d_j(p',q') \ge d_0(p',q')=L_0(\gamma_{p'})
\ee
 These tubes of $g_0$-geodesics are depicted in Figure~\ref{fig-APS-tubes} so that one sees 
how large $L_j(\gamma_{p'})$ when the geodesic reaches into a tip.

By $Vol_j(M) \rightarrow Vol_0(M)$ and \eqref{VolumeSetsConverge} one has
convergence of the volumes of the tubes
\begin{align}\label{FoliationVolume}
    Vol_j(\mathcal{T}_{v,\alpha,p})\rightarrow Vol_0(\mathcal{T}_{v,\alpha,p}).
\end{align}  
They next work to show that if the volumes of the tubes are converging then for almost every
 $\gamma_{p'}$, 
 \be \label{Ljplan}
 L_j(\gamma_{p'}) \to L_0(\gamma_{p'}) \textrm{ and thus by (\ref{71}) one has }
 d_j(p',q') \to d_0(p',q'). 
\ee
To do this rigorously they must be careful to keep track of the variation between the geodesics.
Taking
\begin{align}
 d\exp^{\perp}: N\, \exp_p(S_k) \rightarrow M   
\end{align} 
to be the differential of the normal exponential map where $ N\exp_p(S_k)$ is the normal bundle to $ S_k\subset M$, $|d \exp^{\perp}|_{g_0}$ is the determinant of the map in directions orthogonal to the foliation, $d \mu_{N_{v,\alpha,p}}=d\mu_N$ be the usual measure for $N_{v,\alpha,p} \subset T_pM \approx\R^m$, $\lambda_1^2,...,\lambda_m^2$ the eigenvalues of $g_j$ with respect to $g_0$ where $\lambda_1 \le ... \le \lambda_m$, and $\sqrt{h}$ the square root of the determinant of the metric $h$ for the hypersurface $\exp(N_{v,\alpha,p})$ in normal coordinates on $N_{v,\alpha,p}$ then 
\begin{align}
  \vol_j(\mathcal{T}_{v,\alpha,p})&= \int_{\mathcal{T}_{v, \alpha}}  \sqrt{Det_{g_0}(g_j)}\,dV_{g_0} 
   \\&= \int_{N_{v,\alpha,p}}\int_{\gamma_{p'}}\lambda_1...\lambda_m |d \exp_{p'}^{\perp}|_{g_0}\sqrt{h}dt_{g_0}\,d\mu_{N}
   \\& \ge  \int_{N_{v,\alpha,p}}\int_{\gamma_{p'}}\lambda_1...\lambda_{m-1} |d \exp^{\perp}|_{g_0}\sqrt{h}dt_{g_j}\,d\mu_{N}
    \\&\ge  \int_{N_{v,\alpha,p}}\int_{\gamma_{p'}} |d \exp^{\perp}|_{g_0}\sqrt{h}dt_{g_j}\,d\mu_{N}\label{CrucialVolumeToDistanceEq1} 
   \\&\ge  \int_{N_{v,\alpha,p}}\int_{\gamma_{p'}} |d \exp^{\perp}|_{g_0}\sqrt{h}dt_{g_0}\,d\mu_{N}\label{CrucialVolumeToDistanceEq2} 
   \\&=  \vol_0(\mathcal{T}_{v,\alpha,p}).
\end{align}
By the convergence of the volumes of the tubes in (\ref{FoliationVolume}), one concludes that
 \eqref{CrucialVolumeToDistanceEq1} and \eqref{CrucialVolumeToDistanceEq2} converge to one another:
\be
\int_{N_{v,\alpha,p}} \int_{\gamma_{p'}} |d \exp^{\perp}|_{g_0}\sqrt{h}(dt_{g_j}-dt_{g_0})\,d\mu_{N} \to 0.
 \ee
Next they show that 
\begin{align}
    |d \exp^{\perp}|_{g_0} \ge A_{p,q} >0 \quad \text{on } \mathcal{T}_{v,\alpha,p}, \label{NonDegnerateFoliation}
\end{align}
using a careful discussion to avoid $g_0$-focal points.  Note that the constants $A_{p,q}$ might be
quite small if $p$ and $q$ are almost conjugate to one another. Also, $\sqrt{h} > h_0 > 0$ on $N_{v,\alpha,p}$ since in normal coordinates they notice that $\sqrt{h}=1$ at $p$.

They then obtain (\ref{Ljplan}) rigorously as follows
\begin{align}
    \int_{N_{v,\alpha,p}}&\int_{\gamma_{p'}} |d \exp^{\perp}|_{g_0}\sqrt{h}(dt_{g_j}-dt_{g_0})\,d\mu_{N} 
    \\&\ge A_{p,q} h_0  \int_{N_{v,\alpha,p}} L_j(\gamma_{p'}) - L_0(\gamma_{p'})\,d \mu_N
    \\&\ge A_{p,q} h_0 \int_{N_{v,\alpha,p}} d_j(p',q') - d_0(p',q') \,d \mu_N
    \\&= A_{p,q} h_0 \int_{N_{v,\alpha,p}} |d_j(p',q') - d_0(p',q')| \,d \mu_N, \label{FirstIntegralDistancesToZero}
\end{align}
and hence \eqref{FirstIntegralDistancesToZero} converges to $0$ as $j \rightarrow \infty$.  In particular
for almost every $p' \in \exp_p(N_{v,\alpha,p})$ and $q'$ determined by $p'$ they have $d_j(p',q') \to d_0(p',q')$.  However one needs to show pointwise almost everywhere convergence where $(p',q')\in M\times M$ with $q'$ independent of $p'$
and $p'$ running freely almost everywhere in $M$.

In order to obtain a $M\times M$ open set around $(p,q)$ they need to free themselves from the restrictions to submanifolds depending on $N_{v,\alpha,p}$
and the dependence of $q'$ on $p'$, so  they construct a $2m$ dimensional set of deformations of
$N_{v,\alpha,p}$ as follows:
\be
\mathcal{N}_{p,q}= \{(v,\tau, \eta, w): \,v\in V_\varepsilon, \,\tau\in (-\bar{\tau},\bar{\tau}), \,\eta\in (\bar{\eta}_1, \bar{\eta}_2), 
\, w\in N_{\eta v,\alpha,p_\tau} \}
\ee
with 
\begin{eqnarray*}
v\,&\in &  V_\varepsilon= \{v'\in T_pM:\, |v'|_{g_0}=|v_{p,q}|_{g_0},\, g_0(v',v_{p,q}) >  (1-\varepsilon) |v_{p,q}|_{g_0}\}\\
&& \qquad \qquad \textrm{where }\varepsilon > 0 \textrm{ sufficiently close to 0 depending on $p,q$}\\
\eta\, &\in& (\bar{\eta}_1, \bar{\eta}_2) \textrm{ sufficiently close to 1 depending on $p,q$}\\
\tau\, &\in& (-\bar{\tau},\bar{\tau}) \textrm{ sufficiently close to 0 depending on $p,q$}\\
p_\tau &=& p_{\tau,v}= \exp_p(\tau v) \\
p'_{\tau}&=& p'_{v, \tau, \eta,w}= \exp_{p_\tau}(w) \textrm{ where } w\in N_{v,\alpha,p_\tau}\\
q'_\eta&=&q'_{v, \tau, \eta,w}= \exp_{p'_\tau}(\eta v') \textrm{ after parallel transporting $v$  to $v'=v'_{v, \tau, \eta,w}$}.
\end{eqnarray*}
They define
\be
\Psi_{p,q}: \mathcal{N}_{p,q} \to U_{p,q} \subset M\times M 
\textrm{ to be  }
 \Psi_{p,q}=(p'_{v, \tau, \eta,w}, q'_{v, \tau, \eta,w}) 
\ee
and prove it is bijective onto an open neighborhood $U_{p,q}$ of $(p,q) \in M\times M$ using the fact that
$q$ is not a cut point of $p$.

They repeat the integration as above replacing $N_{v_{p,q},\alpha,p}$ with $N_{v,\alpha,p_\tau}$
where $v\in V_\varepsilon$, $\tau\in (-\bar{\tau},\bar{\tau})$, and $\eta\in (\bar{\eta}_1, \bar{\eta}_2)$.    
In fact the integrals are not only converging to $0$ but also uniformly
bounded above:
\begin{align}
    \int_{N_{ v,\alpha,p}}&\int_{\gamma_{p'_\tau}} |d \exp^{\perp}|_{g_0}\sqrt{h}(dt_{g_j}-dt_{g_0})\,d\mu_{N} 
    \\&\le \vol_j(T_{\eta v,\alpha,p}) - \vol_0(T_{\eta v,\alpha,p})) \le \vol_j(M)\le V_0.
    \end{align}
    Thus they apply the Dominated Convergence Theorem to see that
    \be
    \int_{v\in V_\varepsilon} \int_{-\bar{\tau}}^{\bar{\tau}}
    \int_{\bar{\eta}_1}^{\bar{\eta}_2}
    \int_{N_{v,\alpha,p}} \int_{\gamma_{p'_\tau}} |d \exp^{\perp}|_{g_0}\sqrt{h}(dt_{g_j}-dt_{g_0})\,d\mu_{N} \, d\eta d\tau dV\to 0.
    \ee
    This implies as in (\ref{FirstIntegralDistancesToZero}) that one has
\be
\int_{v\in V_\varepsilon} \int_{-\bar{\tau}}^{\bar{\tau}}
    \int_{\bar{\eta}_1}^{\bar{\eta}_2}
    \int_{N_{v,\alpha,p}}  |d_j(p'_\tau,q'_\eta) - d_0(p'_\tau,q'_\eta)| \,d\mu_{N} \, d\eta d\tau dV\to 0.
    \ee
    Applying the map $\Psi_{p,q}: \mathcal{N}_{p,q} \to U_{p,q}$ we have
    \be  
  \int_{U_{p,q}} |d_j(p',q') - d_0(p',q')| \,d\mu \,d\mu  \to 0.
  \ee
  So there is a subsequence converging pointwise almost everywhere on $U_{p,q}$.
  
  To complete the proof, they observe that $\mathcal{U}=\{(p,q):\, q \notin CutLocus_{g_0}(p)\}$ is a set of full measure
  in $M\times M$  and has a compact exhaustion:
  \be
  K_1\subset K_2 \subset \cdots \subset K_k\subset  \cdots \subset \, \mathcal{U}
  \textrm{ and } \bigcup_{k =1}^\infty K_k \,=\, \mathcal{U}.
  \ee
Since the open cover of each compact set
\be
K_k \, \subset  \, \mathcal{U}\, \subset \bigcup_{(p,q)\in \mathcal{U}} U_{p,q}
\ee
has a finite subcover, we obtain a countable cover of $\mathcal{U}$
\be
\mathcal{U}\subset \bigcup_{i=1}^\infty U_{p_i,q_i}
\textrm{ where each } K_k \subset \bigcup_{i=1}^{N_k} U_{p_i,q_i}.
\ee
They now take a subsequence of $d_j: M \times M \to [0,D]$ which converges pointwise
almost everywhere on $U_{p_1,q_1}$, then a further subsequence which
converges pointwise almost everywhere on $U_{p_2,q_2}$, and so on and diagonalize,
to obtain a subsequence that converges pointwise almost everywhere on all
$U_{p_i,q_i}$ and thus on $\mathcal{U}$ which has full measure in $M\times M$.
\end{proof}
 

\section{A New Explicit Estimate on the Intrinsic Flat Distance}\label{sect:NewSWIFEstimate}

In this section we prove a new explicit estimate on the intrinsic flat distance between
metric spaces where $d_j\ge d_0$ everywhere and $d_j \le d_0+\lambda$ on a set $W$
where $\vol_j(M_j\setminus W_j)$ is small.   This explicit estimate will be applied to prove our
main theorem.

\begin{thm}\label{est-SWIF}
Let $M$ be an oriented, connected and closed manifold, $M_j=(M,g_j)$ and $M_0=(M,g_0)$ be Riemannian manifolds with  $\diam(M_j) \le D$,  $\vol_j(M_j)\le V$
and $F_j: M_j \rightarrow M_0$ a $C^1$ diffeomorphism  and distance non-increasing map:
\be
d_j(x,y) \ge d_0(F_j(x), F_j(y)) \quad \forall x,y \in M_j.
\ee
Let $W_j \subset M_j$, $W_j \not = \emptyset$ be a measurable set and  $\delta_j, V_j, h_j > 0$ so that
 \be\label{eq-distCond}
d_j(x,y) \le d_0(F_j(x), F_j(y)) +2 \delta_j \qquad \forall x,y \in W_j
\ee
with
\be\label{eq-volCond}
\vol_j(M_j \setminus W_j) \le V_j
\ee
and
\be
h_j = \sqrt{2 \delta_j D + \delta_j^2}
\ee
then
\be
d_{\mathcal{F}}(M_0,M_j) \le 2V_j + h_j V.
\ee
\end{thm}

\begin{figure}[h] 
   \center{\includegraphics[width=.4\textwidth]{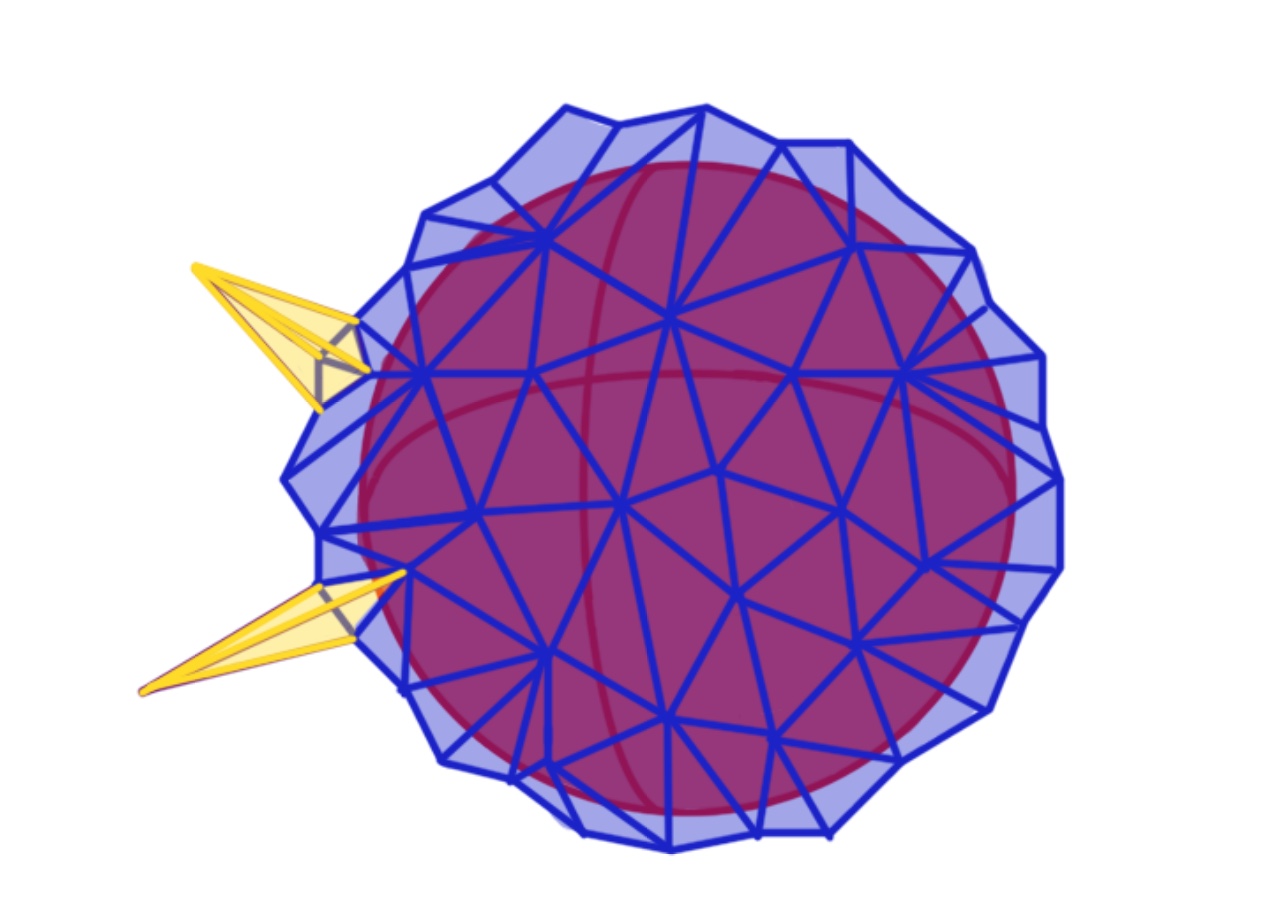}}
\caption{Here we see a piecewise flat $(M,d_j)$ around $(M, d_0)=({\mathbb S}^2, d_{\mathbb S})$ clearly identifying the regions where $d_j$ is not close to $d_0$ in yellow.}
   \label{fig-AS2-A}
\end{figure}

\begin{rmrk} 
Observe that the hypotheses of this theorem are much weaker than the hypotheses of 
the theorem of the third author with Huang and Lee in the Appendix of \cite{HLS} which requires controlling the
distances in biLipschitz way everywhere.  We may also contrast this theorem with an earlier theorem of the third author with Lakzian (see Theorem 4.6 in \cite{Lakzian-Sormani-1}).   The theorem with Lakzian does not require the distance decreasing map we require here, but does require that one obtain uniform bounds on the metric tensor in the good region.  It requires a two-sided distance estimate in place of (\ref{eq-distCond}).  In addition to
a volume estimate similar to (\ref{eq-volCond}), it requires uniform control on the areas of $\partial W_j$.  
All three of these theorems are proven by constructing an explicit common metric space $Z$ into which the oriented manifolds embed via distance preserving maps.  However the metric spaces are quite different for each theorem and thus provide different estimates requiring different bounds. 
\end{rmrk}


\subsection{Constructing the Common Space $Z$}

Now we construct a complete metric space $Z$ for which two Riemannian manifolds can be embedded in a distance preserving manner.

\begin{figure}[h] 
   \center{\includegraphics[width=.8\textwidth]{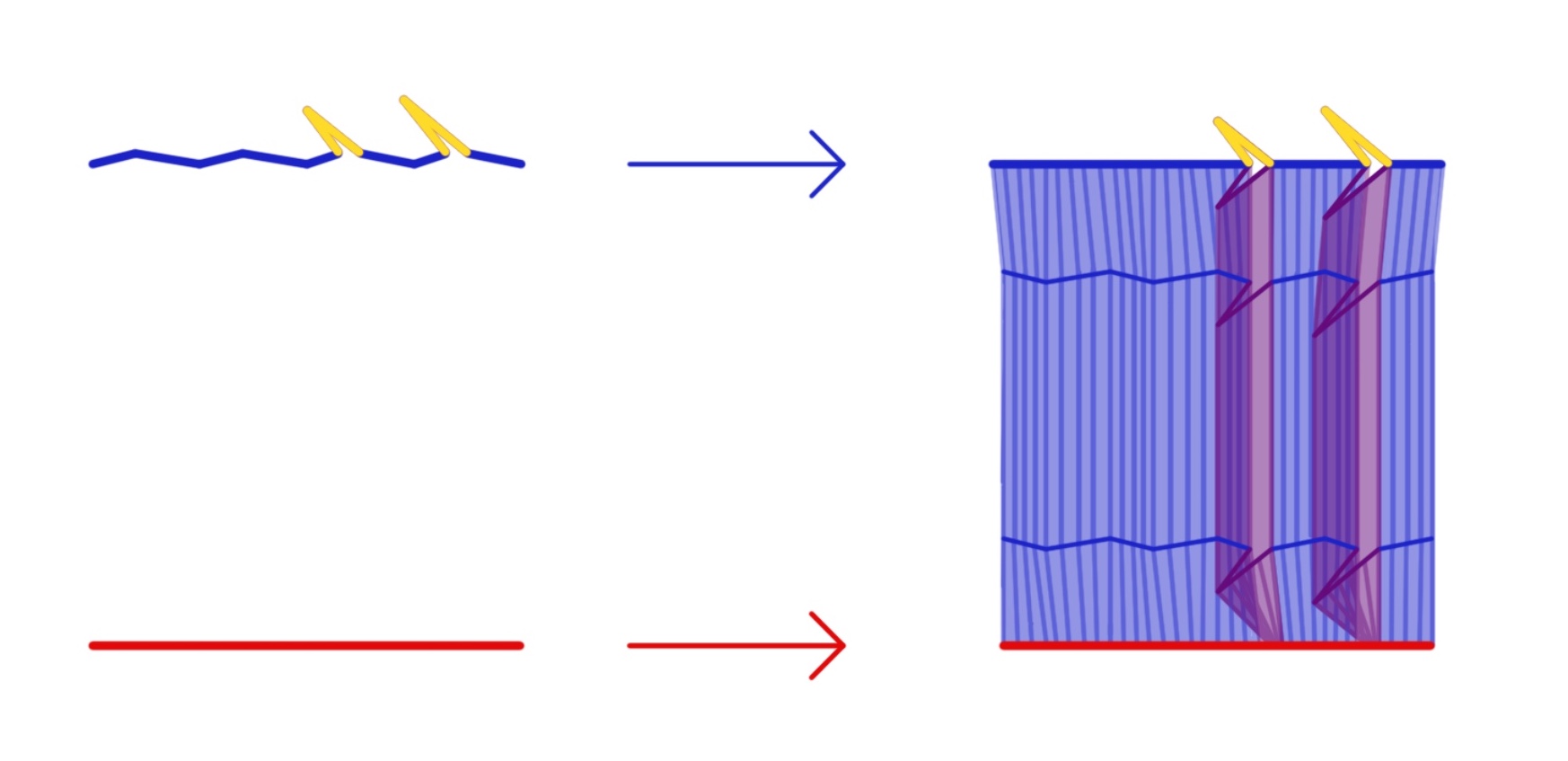}}
\caption{Here we see $(M,d_j)$ and $(M, d_0)$ on the left and $Z$ on the right, using the same coloring as
in Figure~\ref{fig-AS2-A}. The blue part of $(M,d_j)$ represents the good set $W_j$.}
   \label{fig-AS2-B}
\end{figure}

\begin{lem} \label{cnstr0-Z}
Let $M$ be a connected, closed manifold, $M_j=(M,g_j)$ and $M_0=(M,g_0)$ be Riemannian manifolds with  $\diam(M_j) \le D$, and $F_j: M_j \rightarrow M_0$ be a $C^1$ diffeomorphism and distance non-increasing map.   Let $W_j \subset M_j$, $W_j \not = \emptyset$ and define the space $Z$ as in Figure~\ref{fig-AS2-B} to be
\be
Z=  M_0  \disjointunion \left(   M \times [0,h_j] \right) \disjointunion  M_j  \,\,|_\sim
\ee
where we identify points via the bijection
\be
\bar{F}_j: \, M \times \{0\} \subset M\times [0,h_j] \to M_0 \textrm{ where } \bar{F}_j(x,0)=F_j(x),
\ee
and identify points via the bijection
\be
id: \,\overline{W}_j \subset M_j  \to \overline{W}_j \times \{h_j\} \subset M\times [0,h_j] \textrm{ where } id(x)=(x,h_j).
\ee
Then $Z$ is a metric space with distance, $d_Z: Z \times Z \to [0, \infty)$, given by
\be
d_Z(z_1, z_2) = \inf \{L_Z(C):\, C(0)=z_1,\, C(1)=z_2\}
\ee
where $C$ is any piecewise smooth curve whose length, $L_Z$, is determined using $g_j$ in $M_j$,
$g_0$ in $M_0$ and the isometric product $g_j + dh^2$ in $M \times (0,h_j]$. 

In addition, for all points $(x_1,h),(x_2,h') \in M\times[0,h_j] \subset Z$ we have:
\be
d_Z((x_1,h),(x_2,h'))\ge \sqrt{d_0(F_j(x_1),F_j(x_2))^2 +|h-h'|^2} ,
\ee
\be \label{region-dist-dec-to-Z}
d_Z((x_1,h),(x_2,h')) \le \sqrt{d_j(x_1,x_2)^2 +|h-h'|^2}.
\ee 
 \end{lem}

\begin{rmrk}
Note that the way in which we measure the lengths of curves in $M\times(0,h_j] \subset Z$ is via the isometric product $g_j+dh^2$ but we are not claiming that the metric space has a product structure on $M\times(0,h_j]$. In general one does not expect the metric space $(Z,d_Z)$ to have a product structure because it will be advantageous to take advantage of shortcuts through $M_0$ which is identified with $M\times \{0\}$.
\end{rmrk}

\begin{proof}
Observe that the metric space $Z$ constructed in the statement of this lemma is a well defined  length space (see the discussion  of length spaces given in Section 2.1 of Burago, Burago, Ivanov \cite{BBI}). In particular, the set of piecewise smooth curves is a class of admissible paths and we can measure lengths by lengths of admissible paths by using $g_j$ in $M_j$,
$g_0$ in $M_0$ and the isometric product $g_j + dh^2$ in $M \times (0,h_j]$.  Then by Exercise 2.1.2 of \cite{BBI} the distance function $d_Z: Z \times Z \to [0, \infty)$ defined by
\be
d_Z(z_1, z_2) = \inf \{L_Z(C):\, C(0)=z_1,\, C(1)=z_2\}
\ee
 turns $(Z,d_Z)$ into a metric space. Now we would like to show the claimed estimates on $d_Z$.
 
 Given any $(x_1,h_1),(x_2,h_2) \in Z' = M \times [0,h_j] \subset Z$,with the metric on $Z'$ restricted from $Z$, let 
\be
C:[0,1]\to Z \textrm{ such that }C(0)=(x_1,h_1)
\textrm{ and } C(1)=(x_2,h_2).
\ee
We claim that we can take
\be \label{identified-segments}
C([0,1]) \subset Z'=M \times [0,h_j].
\ee

On the contrary, if $C$ were not contained in $Z'$ then let $S \subset [0,1]$ be the maximal subset so that for all $s \in S$, $C(s) \not \in Z'$. Note that by the fact that $C(0),C(1) \in Z'$ we know $S \subset (0,1)$. If we define the map 
\begin{align}
id:M_j \rightarrow M \times \{h_j\}\subset Z, \qquad id(x) = (x,h_j),
\end{align} 
then we can define a new curve 
\begin{align}
\tilde{C}(t)=
\begin{cases}
id(C(t))& \text{ if } t \in S
\\C(t)& \text{ otherwise.}
\end{cases}
\end{align}
By construction $\tilde{C} \subset Z'$ and since we measure the lengths of curves the same way in $M_j$ and in $M\times \{h_j\}$ we find that $L_Z(C([0,1]))=L_Z(\tilde{C}([0,1]))$.

Now assume that $C([0,1]) \subset Z'$ so we can write:
 $C(t)=(x(t), h(t))$ where $x(t)\in M$ and $h(t)\in [0,h_j]$. Then by Lemma \ref{DistToMetric} we know that $g_0 (dF_j(v), dF_j(v))  \leq g_j (v,v)$ and hence 
 \begin{align}\label{LengthLowerBound}
 L_Z(C([0,1])) \ge \int_0^1 \sqrt{g_0(dF_j(x'), dF_j(x'))+ h'^2}dt.
 \end{align}
 Since \eqref{LengthLowerBound} holds for all $C$ and the right hand side is how lengths would be measured in the Riemannian product $g_0 + dh^2$ we can take the infimum over all curves to conclude that
 \begin{align}
 d_Z((x_1,h_1),(x_2,h_2)) \ge \sqrt{d_0(F_j(x_1),F_j(x_2))^2 +|h_1-h_2|^2},
 \end{align}
 for all $(x_1,h_1),(x_2,h_2) \in Z'$.
 
 Again using the fact that $g_0 (dF_j(v), dF_j(v))  \leq g_j (v,v)$ we can observe
 \begin{align}\label{LengthUpperBound}
 L_Z(C([0,1])) \le \int_0^1 \sqrt{g_j(x', x')+ h'^2}dt.
 \end{align}

 Since \eqref{LengthUpperBound} holds for all $C$ and the right hand side is how lengths would be measured in the Riemannian product $g_j + dh^2$ we can take the infimum over all curves  to conclude that
 \begin{align}
 d_Z((x_1,h_1),(x_2,h_2)) \le \sqrt{d_j(x_1,x_2)^2 +|h_1-h_2|^2}.
 \end{align}
\end{proof}

Now we use the metric space, $Z$, constructed in Lemma \ref{cnstr0-Z} to show that $M_j$ and $M_0$ can be embedded in $Z$ in a distance preserving manner.   See Figures~\ref{fig-AS2-A} and~\ref{fig-AS2-B}.

\begin{lem} \label{cnstr-Z}
Let $M$ be a connected, closed manifold, $M_j=(M,g_j)$ and $M_0=(M,g_0)$ be Riemannian manifolds with  $\diam(M_j) \le D$, and $F_j: M_j \rightarrow M_0$ be a $C^1$ diffeomorphism and distance non-increasing map.
Let $W_j \subset M_j$, $W_j \not = \emptyset$, be a measurable set.  Assume that
 \be\label{eq-distCond0}
d_j(x,y) \le d_0(F_j(x), F_j(y)) +2 \delta_j \qquad \forall x,y \in W_j
\ee
and take
\be
h_j = \sqrt {2 \delta_j D + \delta_j^2} 
\ee
then the maps
\be
\varphi_j: M_j \to Z \textrm{ where } \varphi_j(x) = x  \, \textrm{if} \, x \notin \overline{W}_j \, \textrm{otherwise} \, \varphi_j(x) = (x, h_j)
\ee
and
\be
\varphi_0: M_0 \to Z \textrm{ where } \varphi_0(x) = (F_j^{-1}(x), 0)
\ee
are distance preserving maps.
\end{lem}

\begin{proof}  
 First we show that $\varphi_0: M_0 \to Z$ is distance preserving. 
Given any $p,q \in M_0$ where $\varphi_0(p)=(x_p,0), \varphi_0(q)=(x_q,0)$ we can use the estimate of Lemma \ref{cnstr0-Z} to notice
\begin{align}\label{d0distIneq}
d_Z(\varphi_0(p),\varphi_0(q)) \ge d_0(p,q).
\end{align}
Since we can choose a curve $C \subset M \times \{0\}$ whose length achieves the equality in \eqref{d0distIneq} we see that
\begin{align}
d_Z(\varphi_0(p),\varphi_0(q)) = d_0(p,q),
\end{align}
and hence $\varphi_0$ is distance preserving.

\bigskip
Now we show that $\varphi_j: M_j \to Z$ is distance preserving. 
Consider $p,q\in M_j$.  Let
\be
C:[0,1]\to Z \textrm{ such that }C(0)=\varphi_j(p)
\textrm{ and } C(1)=\varphi_j(q).
\ee

In the case where $p,q \in \overline{W}_j$ we know by the proof of Lemma \ref{cnstr0-Z} that we can take
\be
C([0,1]) \subset Z'=M \times [0,h_j],
\ee
with the metric on $Z'$ restricted from $Z$.

Thus we have $C(t)=(x(t), h(t) )$ where
\be \label{isom-prod-parts-2}
x(0)= p \quad h(0)=h_j \quad x(1)=q \quad h(1)=h_j.
\ee
If all of $C([0,1])$ lies above
$h=0$ we have  
\begin{eqnarray}
L_Z(C([0,1]))& = &\int_0^1 \sqrt{ g_j(x'(t),x'(t))+ h'(t)^2 \,} \, dt \\
&\geq& \int_0^1 \sqrt{ g_j(x'(t),x'(t))\,} \, dt =L_{g_j}(x([0,1])).
\end{eqnarray}  

However if $C$ does
reach $h=0$ then we only have
\be\label{eq-trianIn}
L_Z(C([0,1]))\ge d_Z(\varphi_j(p), (x_p,0))+ d_{0}(F_j(x_p), F_j(x_q)) + d_Z((x_q,0), \varphi_j(q))
\ee
where $(x_p,0)$ and $(x_q,0)$ are the first and last points where $C$ hits $h=0$.

By our choice of $h_j$ and  for any $0 < d \le D$ we have,
\be
d^2+h_j^2 \ge  d^2 + 2 \delta_j D + \delta_j^2 \ge  d^2 + 2\delta_j d +\delta_j^2= (d+\delta_j)^2.
\ee
Since $\diam(M_0) \leq \diam(M_j) \leq D$ and using the estimates from Lemma \ref{cnstr0-Z} we find
\begin{align}
d_Z(\varphi_j(p), (x_p,0))  &\ge  \sqrt{d_{0}(F_j(p),F_j(x_p))^2 + h_j^2}\\ 
&\ge d_{0}(F_j(p),F_j(x_p))+ \delta_j
\\ d_Z(\varphi_j(q), (x_q,0)) &\ge \sqrt{d_{0}(F_j(q),F_j(x_q))^2 + h_j^2} 
\\&\ge d_{0}(F_j(q),F_j(x_q))  + \delta_j.
\end{align}
   
Now recall that $F_j$ is distance non-increasing and satisfies
(\ref{eq-distCond0}) where (\ref{eq-distCond0}) also holds for points $p,q$ in the $d_j$ closure of $W_j$ by continuity.
Substituting these observations in (\ref{eq-trianIn}) we find
\begin{align}
L_Z(C[0,1])
&\ge  d_{0}(F_j(p),F_j(x_p)) + d_{0}(F_j(x_p), F_j(x_q)) 
\\&\qquad + d_{0}(F_j(q),F_j(x_q)) + 2\delta_j
\\&\ge  d_{0}(F_j(p),F_j(q)) + 2\delta_j 
\\&\ge d_{j}(p,q).
\end{align}
Since we can choose a $C \subset Z'$ which realizes the distance $d_j(p,q)$ we see that $\varphi_j$ is distance preserving for $p,q\in \overline{W}_j$.

If $p$ or $q$ lies in $M_j \setminus \overline{W}_j$, then any curve $C:[0,1]\to Z$ from
$C(0)=\varphi_j(p)$ to $C(1)=\varphi_j(q)$ starts and ends at a point which is not
identified with a point in $Z'$. If no points in $C$ are identified with a point in $Z'$
then 
\be
L_Z(C[0,1])=L_{g_j}(C[0,1]) \ge d_{j}(p,q).
\ee
Otherwise let $p'$ be the first point on $C$ identified with a point in $Z'$ and
$q'$ be the last such point.  Then $p', q' \in \overline{W}_j$, and so we know from above that
\be
d_Z((p',h_j),(q',h_j))=d_{j}(p',q').
\ee
Applying the fact that unidentified points are measured using $d_{j}$
we have
\be
L(C[0,1])\ge d_{j}(p,p')+ d_{j}(p',q')+d_{j}(q',q)\ge d_{j}(p,q).
\ee
Since we can choose a $C \subset M_j$ which realizes the distance $d_j(p,q)$ we see that $\varphi_j$ is distance preserving for $p,q\in M_j \setminus\overline{W}_j$. Hence $\varphi_j:M_j \to Z$ is distance preserving, as desired.
\end{proof}


\subsection{Estimating Intrinsic Flat Distance}

We now use the metric space $Z$ constructed in the previous subsection in order to give a new estimate on the intrinsic flat distance between Riemannian manifolds.  Readers may wish to review Subsection~\ref{subsect:SWIFBackground} before reading this proof.

\begin{proof}[Proof of Theorem \ref{est-SWIF}]
In order to estimate the intrinsic flat distance between $M_j$ and $M_0$ we must be very careful with orientation.
Remember $M_j=(M,g_j)$ and $M_0=(M, g_0)$ where $M$ is the same compact oriented manifold and $F_j: M_j\to M_0$
is biLipschitz.  So there is
an oriented atlas of smooth charts 
\be
\phi_i:  U_i  \subset \R^m  \to \phi_i(U_i)\subset M_j \textrm{ and } F_j\circ \phi_i: U_i  \subset \R^m  \to F_j(\phi_i(U_i))\subset M_0.
\ee  
Note that these charts are diffeomorphisms so they are biLipschitz with different constants for both
$M_j=(M,g_j)$ and $M_0=(M,g_0)$ and they can be restricted to $A_j \subset U_j$
to ensure they are pairwise disjoint as required when considering them as rectifiable charts for $M_j$ and $M_0$.
Furthermore
\begin{align}\label{eq-canonicalT}
[[M_j]] (f,\pi)= & \sum_{i=1}^\infty  \int_{ A_i}  (f \circ \phi_i )  \, d(\pi_1\circ \phi_i)\wedge\cdots \wedge d(\pi_m\circ \phi_i) 
\end{align}
for any $f:  M_j  \to \R$  Lipschitz and bounded and $\pi=(\pi_1,...,\pi_m)$ where each component is Lipschitz
and
\begin{equation}
[[M_0]]=  {F_j}_\sharp[[M_j]].   
\end{equation} 

Let $\iota: [0,h_j]  \to [0,h_j]$ be the identity map. Then, $(\phi_i, \iota) :  A_i\times [0,h_j]  \to M_j  \times [0,h_j]$ defines an oriented atlas of biLipschitz maps. Then we can write $[[ \,  M_j \times [0,h_j]  \, ]]$ as a countable sum of integrals as above using this atlas. 

Now consider the identity map $\iota_j: M_j \times [0,h_j]   \to   Z'_j$. Since it  is $1$-Lipschitz by (\ref{region-dist-dec-to-Z}) and bijective,  the maps 
$$\iota_j \circ    (\phi_i, \iota) :   A_i\times [0,h_j]      \to      M_j  \times [0,h_j]     \to  Z'$$    
define an oriented atlas of Lipschitz maps  for $Z'_j$,   where the maps can be considered to be biLipschitz as before. 
Thus,  we can define the current with weight $1$ given by this oriented atlas, $B$. Moreover, 
\begin{equation}
B= {\iota_j}_\sharp [[  \,M_j \times [0,h_j] \,]]. 
\end{equation}
Recall that the boundary operator commutes with the pushforward operator, thus
\begin{equation}
\partial B= {\iota_j}_\sharp \partial[[ \, M_j \times [0,h_j] \,]]= {\iota_j}_\sharp    \alpha_\sharp [[ M_j  \times \{h_j \}  ]]  -   {\iota_j}_\sharp \beta_\sharp   [[M_j \times \{0\}]],     
\end{equation}
where $\alpha: M_j \times \{h_j\}   \to  M_j \times [ 0, h_j]$  and $\beta: M_j \times \{0\}   \to  M_j \times [ 0, h_j]  $ are inclusion maps and are trivially Lipschitz maps. 

By the definition of $\varphi_0$, 
\begin{eqnarray}
\varphi_{0\#}[[ M_0]]  &  =  &     {\iota_j}_\sharp \beta_\sharp   [[M_j \times \{0\}]]. 
\end{eqnarray}

Since $W_j \subset M_j$ is a measurable set, we can define an integer rectifiable current of weight $1$, $[[W_j  \times \{h_j \}  ]]$,  by restricting the atlas of $M_j \times \{h_j\}$.   In a similar way,  $[[M_j \setminus W_j]]$ is a well define integer rectifiable current. 

By the definition of $\varphi_j$, and
\begin{eqnarray}
\varphi_{j\#}[[ M_j]]  &  =  &   {\iota_j}_\sharp    \alpha_\sharp [[W_j  \times \{h_j \}  ]]  +     \tilde \alpha_\sharp [[M_j \setminus W_j]].  
\end{eqnarray}

We define now an integer rectifiable current in the following way, 
\begin{equation}
A  =   \tilde \alpha_\sharp [[M_j \setminus W_j]]  -     {\iota_j}_\sharp    \alpha_\sharp [[   (M_j \setminus W_j  ) \times \{h_j \} ]],
\end{equation}
where $\tilde \alpha: M_j \setminus W_j  \to Z$ is the inclusion map, which is Lipschitz since it is distance preserving. 
Note that the second term in $A$ corresponds to the current of weight $1$ on the set of unidentified points in $Z$ drawn in yellow
in Figure~\ref{fig-AS2-B}. 
Furthermore, $A$ is an integral current given that 
\begin{equation}
\partial  A  =      \partial  \tilde \alpha_\sharp [[M_j \setminus W_j]]  -  \partial   {\iota_j}_\sharp    \alpha_\sharp [[   (M_j \setminus W_j  ) \times \{h_j \} ]]=0.
\end{equation}

From the previous equalities, 
\begin{equation}
 A = \varphi_{j\#}[[M_j]]   -  {\iota_j}_\sharp    \alpha_\sharp [[ M_j \times \{h_j \} ]].
 \end{equation}
We conclude that 
\be
\partial B + A = \varphi_{j\#}[[M_j]]- \varphi_{0\#}[[ M_0]]
\ee
and thus 
\be
d_{\mathcal{F}}(M_j, M_0) \le \mass(B) + \mass(A).
\ee

To finish the proof, since $B=  {\iota_j}_\sharp [[ \,  M_j \times [0,h_j]   \,  ]]$ we know that 
\be
\mass(B) \leq  \mass([[\, M_j \times [0,h_j] \,]])  = \vol_j(M_j \times [0,h_j]) \leq h_jV, 
\ee
where we used the fact that the map $\iota_j:  M_j \times [0,h_j]  \to Z'$ is a 1-Lipschitz map.  
In a similar way, 
\be
\mass(A)  \leq 2 \vol_j(M_j \setminus W_j).
\ee
\end{proof}


\section{Pointwise Convergence and Volume Bounds imply Intrinsic Flat Convergence}\label{sect:GoodSet}

In this section we prove the following theorem which we will
later apply  to prove our main theorem.

\begin{thm} \label{ptwise-to-SWIF}
Suppose we have a fixed closed and oriented Riemannian manifold, $M_0=(M,g_0)$,
and  a sequence of metric tensors $g_j$ on $M$ defining $M_j=(M, g_j)$ such that
\be 
g_0(v,v) \le g_j(v,v), \quad \forall v \in T_pM,
\ee
\be
\diam_j(M_j) \le D_0,
\ee
\be\label{ptwise-to-bulk-1}
d_j(p,q)\to d_0(p,q) \textrm{ pointwise a.e. } p,q \in (M\times M, d_0\times d_0)
\ee
and 
\be 
\vol_j(M_j) \rightarrow \vol_0(M_0).
\ee
 Then
\be
d_{\mathcal{F}}(M_j, M_0 ) \to 0.
\ee
\end{thm}

Theorem \ref{ptwise-to-SWIF} is proven by ensuring that we can apply Theorem \ref{est-SWIF}.  In particular, we need to show the existence of subsets in $M$ satisfying
  \eqref{eq-distCond} and \eqref{eq-volCond}.  We now give an outline of this.  In Subsection \ref{subsect:Egoroff's Theorem} we use Egoroff's theorem to go from pointwise convergence almost everywhere to uniform converence on a set of almost full measure $S_\vare \subset M\times M$.  This set is not contained in $M$ and hence cannot be used to apply Theorem \ref{est-SWIF}.   As a preliminary step,  in Subsection \ref{subsect:Fubini} we 
use the coarea formula to see that for almost every $p \in M$,  the sets $S_{p,\vare} = \{ q \, |\,  (p, q)  \in S_\vare \} \subset M$ have almost full measure.  
By Egoroff's theorem we only know that  for all $q \in S_{p,\vare}$  we have $d_j(p,q) \to d_0(p,q)$. Thus, in Subsection \ref{subsect:Good Set} we define the good sets $W_{\kappa \varepsilon}$, that are used to apply  Theorem \ref{est-SWIF} as the set of points $p$  such that $S_{p,\vare}$ have almost full measure (quantified in terms of $\kappa$), and we show that these $W_{\kappa \varepsilon}$ also have almost full measure (quantified in terms of $\kappa$).  That is, they satisfy  \eqref{eq-volCond}.    In Subsection \ref{subsect:Distance Bounds} we ensure that the good sets  satisfy  \eqref{eq-distCond}.  Notice that given two points $p_1, p_2  \in W_{\kappa \varepsilon}$ we need to show that $d_j(p_1, p_2)$ close to $d_0(p_1, p_2)$ in a specific quantified way. Since $(p_1,p_2)$  might not be contained in $S_{p_1,\vare}$ we show the existence of a point $q \in S_{p_1,\vare} \cap S_{p_2,\vare}$ so that we can use a triangle inequality argument to 
obtain the uniform distance bound on pairs of points contained in the good set  $W_{\kappa \varepsilon}$. In Subsection \ref{subsect:Proof of SWIF} we finish the proof of Theorem \ref{ptwise-to-SWIF} by applying Theorem \ref{est-SWIF} in combination with all previous subsections.

\subsection{Egoroff's Theorem}\label{subsect:Egoroff's Theorem}
We begin by reminding the reader of Egoroff's theorem which can be found in the book of Evans and Gariepy \cite{Evans-Gariepy}.

\begin{thm}[Egoroff's Theorem] \label{Egoroff's Theorem}
Let $f_n:X \rightarrow \R$ be a sequence of measurable functions on a measure space $(X, \mu)$. Assume there is a measurable set $A \subset X$, $\mu(A) < \infty$, so that $f_n$ converges pointwise $\mu-$almost everywhere to a measurable function $f$. Then for every $\varepsilon > 0$, there exists a measurable subset $B_{\varepsilon} \subset A$ so that 
\begin{align}
\mu(B_{\varepsilon} ) < \varepsilon
\end{align}
and 
\begin{align}
f_n \rightarrow f
\end{align}
uniformly on $A \setminus B_{\varepsilon} $.
\end{thm}

Now we apply Egoroff's theorem to obtain uniform convergence on a set of almost full measure.

\begin{prop}\label{Svare}
Under the hypotheses of  Theorem \ref{ptwise-to-SWIF},  for every $\vare >0$ there exists a
$dvol_{g_0}\times dvol_{g_0}$ measurable set, $S_\vare\subset M\times M$, such that
\be\label{unifSvare}
\sup\{|d_j(p,q)- d_0(p,q)|\,:\, (p,q)\in S_\vare\}=\delta_{\vare,j} \to 0,
\ee
\be\label{volSvare}
\vol_{0\times 0} (S_\vare)> (1-\vare)\vol_{0\times 0}(M\times M).
\ee
and 
\be \label{Svaresym}
(p,q)\in S_\vare \iff (q,p) \in S_\vare.
\ee
\end{prop}

\begin{proof}
By Egoroff's Theorem \ref{Egoroff's Theorem} since $(M,d_0,dvol_{g_0})$ is a metric measure space so that $dvol_{g_0}(M)< \infty$ and
\be
d_j(p,q)\to d_0(p,q) \textrm{    ptwise  $dvol_{g_0} \times dvol_{g_0}$ a.e. } (p,q) \in M\times M,
\ee
then for all $\vare>0$ there exists a $dvol_{g_0}\times dvol_{g_0}$ measurable set, $S_\vare\subset M\times M$, such that
\be\label{unifSvare}
\sup\{|d_j(p,q)- d_0(p,q)|\,:\, (p,q)\in S_\vare\}=\delta_{\vare,j} \to 0
\ee
and
\be\label{volSvare}
\vol_{0\times 0} (S_\vare)> (1-\vare)\vol_{0\times 0}(M\times M).
\ee
Note that since $d_j(p,q)=d_j(q,p)$ we can ensure, by possibly enlarging $S_\vare$, that
\be \label{Svaresym}
(p,q)\in S_\vare \iff (q,p) \in S_\vare.
\ee
\end{proof}


\subsection{Product Structures} \label{subsect:Fubini}   

We now use the product Riemannian structure on $(M\times M, g_0 \times g_0)$ in order to relate $S_{\varepsilon} \subset M \times M$ to subsets of $M$ through the control on the volume of $S_{\varepsilon}$.

\begin{lem}\label{lem-Spvare}
Under the assumptions of Proposition \ref{Svare},  for almost every $p \in M$ the sets
\be\label{Spvare}
S_{p,\vare}=\{q\in M\,:\, (p,q)\in S_\vare\},
\ee
are $dvol_{g_0}$ measurable and satisfy 
\be\label{AverageAreaOfGoodSetInequality}
(1-\vare) \vol_0(M) < \int_{p\in M} \frac{\vol_0(S_{p,\vare})}{ \vol_0(M)}\, dvol_{g_0}.
\ee
\end{lem}

See Figure~\ref{fig-AS-star}.

\begin{figure}[h] 
   \center{\includegraphics[width=.8\textwidth]{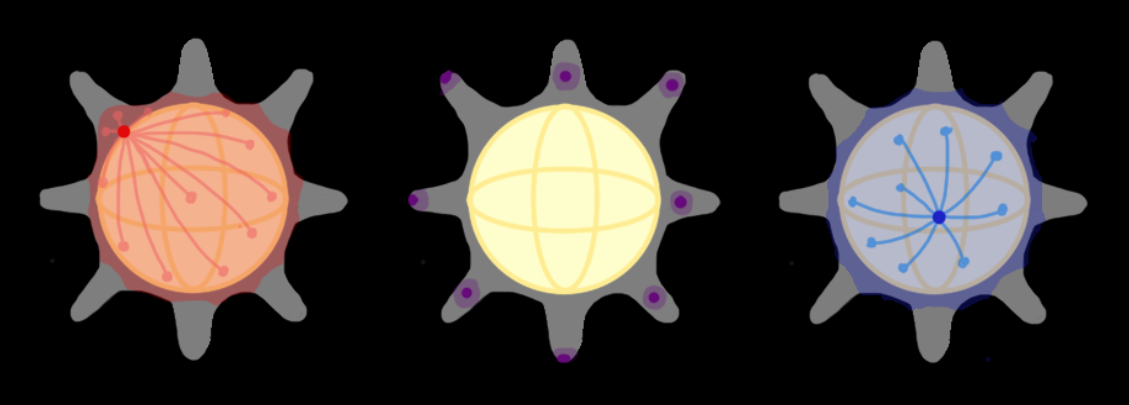}}
\caption{Here we see three copies of $(M,d_j)$ with volume close to $(M,d_0)=({\mathbb S}^2, d_{\mathbb{S}})$.   
On the left we see a point $p_1$ at the base of a well whose $S_{p_1,\vare}$ is most of the
manifold except the other wells.  On the right we see a point $p_2$ far from a well whose $S_{p_2,\vare}$ is most
of the manifold away from the wells.  In the middle we see points $p_i$ on wells whose $S_{p_i,\vare}$ is small.}
   \label{fig-AS-star}
\end{figure}

\begin{proof} 
Since $S_{\varepsilon}$ is $dvol_{g_0}\times dvol_{g_0}$ measurable it follows that for almost every $p$ such that 
$(p,q) \in  S_{\varepsilon}$ for some $q$, 
\begin{align}
S_{p,\vare}
\end{align}
is measurable.

Moreover, by (\ref{Svaresym}) we have
\be \label{Spvaresym}
q\in S_{p,\vare} \iff p \in S_{q,\vare}
\ee

Now by the product Riemannian structure on $(M \times M, g_0 \times g_0)$:
\be
\vol_{0\times 0} (S_\vare)=\int_{p\in M} \vol_0(S_{p,\vare}) \,dvol_{g_0}. 
\ee
Thus, by (\ref{volSvare}) and $\vol_{0\times0}(M\times M)=(\vol_0(M))^2$, we get
\be\label{AverageAreaOfGoodSetInequality1}
(1-\vare) \vol_0(M) < \int_{p\in M} \frac{\vol_0(S_{p,\vare})}{ \vol_0(M)}\, dvol_{g_0}.
\ee
\end{proof}


\subsection{Selecting our Good Set}\label{subsect:Good Set}

For $\vare >0$ and $\kappa >1$ such that $\kappa \vare < 1$ let
\be \label{Wkappavare}
W_{\kappa\vare}=\{p:\, \vol_0(S_{p,\vare}) > (1- \kappa\vare) \vol_0(M)\}.
\ee

First we notice that $W_{\kappa \vare}$ is measurable by defining the function  $\Phi:M \rightarrow [0,\infty)$ as $\Phi(p) = \int_M \mathbbm{1}_{S_{\vare}}(p,q) dvol_{g_0}(q)$ and so $W_{\kappa \vare}$ is measurable since it is the preimage of a measurable function.

In Figure~\ref{fig-AS-star} we can intuitively see that $W_{\kappa\vare}$
consists of points like $p_1$ and $p_2$ that do not lie inside the wells. In the following lemmas we show that $W_{\kappa \vare}$ has the correct volume to be used as the good set in Theorem \ref{est-SWIF}.

\begin{lem}\label{vol-W} For $W_{\kappa\vare}$ defined as in \eqref{Wkappavare} we find
\be
\vol_0(W_{\kappa\vare}) > \frac{\kappa-1}{\kappa} \vol_0(M).
\ee
\end{lem}

\begin{proof}
Starting with \eqref{AverageAreaOfGoodSetInequality} calculate
\begin{eqnarray*}
 (1-\vare) \vol_0(M) &<& \int_{p\in W_{\kappa\vare}} \frac{\vol_0(S_{p,\vare})}{ \vol_0(M)}\, dvol_{g_0}
+ \int_{p\in M\setminus W_{\kappa\vare}} \frac{\vol_0(S_{p,\vare})}{\vol_0(M)}\, dvol_{g_0}\\
&\le &\int_{p\in W_{\kappa\vare}} 1 \, dvol_{g_0}
+ \int_{p\in M\setminus W_{\kappa\vare}} (1- \kappa\vare)\, dvol_{g_0}\\
&=& \vol_0(W_{\kappa\vare}) +
(1- \kappa\vare) \vol_0(M\setminus W_{\kappa\vare})\\
&=& \vol_0(W_{\kappa\vare}) +
(1- \kappa\vare) \vol_0(M) -(1- \kappa\vare)\vol_0(W_{\kappa\vare})\\
&=&\kappa\vare \, \vol_0(W_{\kappa\vare})+ (1- \kappa\vare) \vol_0(M).
\end{eqnarray*}
Hence,  
\begin{equation}
(\kappa\vare-\vare) \vol_0(M) < \kappa\vare \, \vol_0(W_{\kappa\vare}). 
\end{equation}
This concludes the proof. 
\end{proof}

\begin{lem}\label{vol_j-W}
For $W_{\kappa\vare}$ defined as in \eqref{Wkappavare} we get 
\be
\vol_j(M \setminus W_{\kappa\vare}) \le \frac{1}{\kappa}\vol_0(M)+ |\vol_j(M)-\vol_0(M)|.
\ee
\end{lem}

\begin{proof}

From $g_0 \le g_j$ we know that $d_0 \leq d_j$. Then, $\vol_0 \leq \vol_j$ and the following holds, 
\begin{equation*}
\vol_0(W_{\kappa \vare}) + \vol_j(M) \leq   \vol_j(W_{\kappa \vare}) + \vol_j(M) +  \left(\vol_0(M) -  \vol_0(M) \right).
\end{equation*}
Rearranging terms, 
\begin{equation*}
-   \vol_j(W_{\kappa \vare}) +  \vol_j(M) \leq  - \vol_0(W_{\kappa \vare}) + \vol_0(M) +  \left(\vol_j(M) -  \vol_0(M) \right).
\end{equation*}
Then by Lemma \ref{vol-W}, 
$$
 \vol_j(M \setminus W_{\kappa \vare}) <  - \frac{\kappa-1}{\kappa} \vol_0(M) + \vol_0(M) +  \left(\vol_j(M) -  \vol_0(M) \right).
$$

\end{proof}


\subsection{Uniform Distance Bounds}\label{subsect:Distance Bounds}

The aim of this subsection is to prove  Lemma \ref{unif-on-W} where we find a 
uniform distance bound on pairs of points contained in the good sets $W_{\kappa \varepsilon}$. 
More explicitly, given two points $p_1, p_2  \in W_{\kappa \varepsilon}$ we need to show that $d_j(p_1, p_2)$ is close to $d_0(p_1, p_2)$ in a specific quantified way,  see \eqref{eq-lemUnifW}.
Since in principle $(p_1,p_2)$  might not be contained in $S_{p_1,\vare}$, we show that $S_{p_1,\vare} \cap S_{p_2,\vare} \neq \emptyset$ and then use a triangle inequality argument to get  Lemma \ref{unif-on-W}.  In Figure~\ref{fig-AS-star} note that $S_{p_1,\vare} \cap S_{p_2,\vare}$  would consist of everything not in the wells and thus has a large volume and for points there their $d_j$ and $d_0$ distances are almost the same.

\begin{lem} \label{Sp12int}
Consider $W_{\kappa\vare}$ defined as in \eqref{Wkappavare} and $S_{p,\epsilon}$ defined in \eqref{Spvare}.
Let $p_1, p_2$ be two points in $W_{\kappa\vare}$. Then
$S_{p_1,\vare}$ and $S_{p_2,\vare}$ cannot be disjoint for $\kappa \vare < 1/2$.
In fact,
\be
\vol_0(S_{p_1,\vare} \cap S_{p_2,\vare}) > (1-2\kappa\vare) \vol_0(M).
\ee
\end{lem}

\begin{proof}
If they are disjoint then
\be
\vol_0(S_{p_1,\vare}) + \vol_0(S_{p_2,\vare} )\le \vol_0(M). 
\ee
Then by  (\ref{Wkappavare}) we get 
\be
(1- \kappa\vare) \vol_0(M)+(1- \kappa\vare) \vol_0(M)< \vol_0(M)
\ee
which implies $(1-\kappa \vare)  \leq  1/2$.   

In fact, taking $K_i=M\setminus S_{p_i,\vare}$ we have
\begin{eqnarray}
\vol_0(S_{p_1,\vare} \cap S_{p_2,\vare}) &\ge& \vol_0(M)-\vol_0(K_1\cup K_2)\\
&\ge & \vol_0(M)-(\vol_0(K_1) +\vol_0(K_2))\\
&>  & \vol_0(M)(1 -\kappa \vare-\kappa\vare).
\end{eqnarray}
\end{proof}

\begin{lem} \label{ball-in-0}
Let $M_0$ be a compact Riemannian manifold. For any $\lambda'  \in (0, \diam(M_0))$, $\kappa >1$ there exists $\vare > 0$ small enough so that $\kappa\vare \in (0,1/2)$ and
\be
\min_{x\in M} \vol_0(B(x,\lambda')) \geq 2\kappa\vare \vol_0(M)
\ee
and thus under the hypotheses of Lemma \ref{Sp12int}, 
\be
B(x,\lambda') \cap S_{p_1,\vare} \cap S_{p_2,\vare} \neq \emptyset \quad \forall x \in M, \,\, p_1,p_2\in W_{\kappa\vare}.
\ee
\end{lem}

Note that $\kappa\vare$ is a decreasing function of $\lambda'$.
In Figure~\ref{fig-AS-star} note that $S_{p_1,\vare} \cap S_{p_2,\vare}$
would consist of everything not in the wells and any point $x$ in $M$
cannot be far away when measured using $d_0$.  Note that $x$ lying on the
tip of a well might be far away measured using $d_j$.

\begin{proof}
Observe that there is some $K$ possibly negative, such that the Ricci curvature on
$(M, g_0)$ has $Ric(g_0) \ge (m-1)K$ where $m$ is the dimension of $M$.
By the Volume Comparison Theorem
we know that for $r_1 \leq r_2$
\begin{align}
    \frac{\vol_0(B(x,r_1))}{\vol_K(B(x_K,r_1))} \ge \frac{\vol_0(B(x,r_2))}{\vol_K(B(x_K,r_2))},
\end{align}
where  $B(x_K,r_1)$ is a ball in the $m$ dimensional space form of constant sectional curvature $K \in \R$, and $\vol_K$ is the volume as measured in this space form. Now by choosing $r_2 = \diam(M_0)$,  we find
\begin{align}
    \vol_0(B(x,r_1)) \ge \frac{\vol_K(B(x_K,r_1))}{\vol_K(B(x_K,\diam(M_0)))} \vol_0(M) \quad \forall x \in M.
\end{align}
Hence by choosing $r_1 = \lambda'$, let $\vare > 0$ be chosen so that the equality holds
\begin{align}
    \frac{\vol_K(B(x_K,\lambda'))}{\vol_K(B(x_K,\diam(M_0)))} =  2\kappa \vare.
\end{align}
Thus we get the result.
\end{proof}

\begin{lem}\label{unif-on-W}
Let $M_j,M_0$ be Riemannian manifolds which satisfy the hypotheses of Theorem \ref{ptwise-to-SWIF}. For any $\lambda'>0$ and $\kappa>1$, let $\vare >0$ be given as
in Lemma~\ref{ball-in-0}.  Then, for $\delta_{\vare,j}$ defined as in Proposition \ref{Svare}, we find 
\be\label{eq-lemUnifW}
|d_j(p_1,p_2)-d_0(p_1,p_2)| < 2 \lambda' + 2\delta_{\vare,j}
\ee
for all $p_1, p_2 \in W_{\kappa\vare}$.
\end{lem}

Note that for $(x_1,x_2) \in S_\vare$ by (\ref{unifSvare}) we have a better distance bound but we are calculating a distance  bound for points in 
$W_{\kappa\vare} \times  W_{\kappa\vare}$ which is not necessarily contained in $S_\vare$. This happens, in particular, when $p_2 \notin S_{p_1,\vare}$.

\begin{proof}
 Let $p_1, p_2 \in W_{\kappa\vare}$, and let $x$ be their $d_0$-midpoint:
\be\label{midpt}
d_0(p_1,x)+d_0(p_2,x)= d_0(p_1,p_2).
\ee
By Lemma~\ref{ball-in-0}, there exists
\be
q\in B_0(x,\lambda') \cap S_{p_1,\vare} \cap S_{p_2,\vare}.
\ee
So
\be
(p_i,q) \in S_\vare \textrm{ and } d_0(x,q)<\lambda'.
\ee
By (\ref{unifSvare})  
\be
d_0(p_i,q)\le d_j(p_i,q) < d_0(p_i,q)+\delta_{\vare,j}.
\ee
Combining this with the triangle inequality, we have
\begin{eqnarray}
\qquad d_0(p_1,p_2)&\le& d_j(p_1,p_2)\\
&\le& d_j(p_1,q)+d_j(q,p_2)\\
&\le & d_0(p_1,q)+d_0(q,p_2) + 2 \delta_{\vare,j}\\
&\le & (d_0(p_1,x) + d_0(x,q))+ (d_0(q,x)+d_0(x, p_2)) + 2 \delta_{\vare,j}\\
& < & (d_0(p_1,x) + \lambda') + (\lambda'+d_0(x,p_2)) + 2 \delta_{\vare,j}\\
&=& d_0(p_1,p_2)+ 2 \lambda' + 2\delta_{\vare,j}
\end{eqnarray}
where the last line follows from (\ref{midpt}). 
\end{proof}

\subsection{Proof of Theorem \ref{ptwise-to-SWIF}}\label{subsect:Proof of SWIF}

\begin{proof}[Proof of Theorem \ref{ptwise-to-SWIF}]
For any $\kappa >1$ and $\lambda' \in  (0, \diam(M_0))$, let $\vare>0$ be given as in 
Lemma \ref{ball-in-0}. That is, choose $\vare > 0$ such that 
\begin{align}
    \frac{\vol_K(B(x_K,\lambda'))}{\vol_K(B(x_K,\diam(M_0)))} =  2\kappa \vare. 
\end{align}
Thus, with this $\vare >0$ we obtain by applying the results in sections  \ref{subsect:Fubini} - \ref{subsect:Good Set}
a set $S_\vare$ and  a set $W_{\kappa \vare}$,  see (\ref{Spvare}) and (\ref{Wkappavare}), such that by Lemma \ref{vol_j-W}, 
\be
\vol_j(M \setminus W_{\kappa \vare}) \le  \frac{1}{\kappa}\vol_0(M)+|\vol_j(M)-\vol_0(M)|.
\ee
Moreover, by Lemma \ref{unif-on-W} we find that 
\be
|d_j(p_1,p_2)-d_0(p_1,p_2)| < 2 \lambda' + 2\delta_{\vare,j}
\ee
for all $p_1, p_2 \in W_{\kappa\vare}$.

Thus, we can apply Theorem \ref{est-SWIF} to get
\begin{align}
d_{\mathcal{F}}(M_0,M_j) \le   \left( \tfrac{2}{\kappa}\vol_0(M)+  2|\vol_j(M)-\vol_0(M)|   +   h_j  V  \right)
\end{align}
where $h_j =   \sqrt{  2(\lambda'+\delta_{\varepsilon,j})D  +   (\lambda'+\delta_{\varepsilon,j})^2}$
and $V >0$ is an upper volume bound which exists since $\vol_j(M) \to \vol_0(M)$ by hypothesis. 

Hence we find
\begin{align}
\limsup_{j\rightarrow \infty} d_{\mathcal{F}}(M_0,M_j) \le  \left(  \tfrac{2}{\kappa} \vol_0(M)+   \sqrt {2 \lambda'D  +   \lambda'^2}\, V \right),
\end{align}
and since this is true for any $\kappa, \lambda'$ we find that
\begin{align}
\limsup_{j\rightarrow \infty} d_{\mathcal{F}}(M_0,M_j) =0.
\end{align}
\end{proof}

\section{Proving our Main Results}\label{sect:ProofMainThm}

In this section we will combine our results to prove Theorem~\ref{vol-thm} which was stated in the introduction.
As a corollary to Theorem \ref{vol-thm} we notice that we are allowed to loosen the metric inequality from below in \eqref{g_j-below-vol-thm} and still come away with the same conclusion. This is useful in applications such as the geometric stability of the scalar torus rigidity theorem explored by Allen, Hernandez-Vazquez,Parise, Payne, and Wang \cite{AHMPPW1} in the case of warped products and Cabrera Pacheco, Ketterer, and Perales \cite{CPKP19} in the case of graphs. The following corollary  has been  applied by Allen to prove geometric stability of the scalar torus rigidity theorem in the conformal case  \cite{Allen-Conformal-Torus}.

\begin{cor}\label{cor-vol-thm}
Suppose we have a fixed  compact oriented Riemannian manifold, $M_0=(M,g_0)$,
without boundary and
a sequence of metric tensors $g_j$ on $M$ defining $M_j=(M, g_j)$ with
\be \label{g_j-below-vol-thm}
\left(1 - \tfrac{1}{2j} \right)g_0(v,v) \le g_j(v,v) \qquad \forall v\in T_pM
\ee 
and a uniform upper bound on diameter
\be
\diam_j(M_j) \le D_0
\ee
and volume convergence
\be
\vol_j(M_j) \to \vol_0(M_0)
\ee
then 
\be
M_j \VFto M_0.
\ee
\end{cor}

\subsection{Proof of Theorem~\ref{vol-thm}}

\begin{proof}

By Theorem \ref{PointwiseConvergenceAE} there exists a subsequence such that 
\be
\lim_{j\to \infty} d_{j}(p,q) = d_0(p,q) \textrm{ pointwise a.e. } (p,q) \in M\times M.
\ee
Hence by combining the hypotheses of Theorem \ref{vol-thm} with Theorem \ref{ptwise-to-SWIF} we have
\be
d_{\mathcal{F}}(M_j, M_0 ) \to 0.
\ee
If not, we would have a subsequence so that
\begin{align}
d_{\mathcal{F}}(M_{j_i},M_0 )> \varepsilon,
\end{align}
then by the argument above there would be a subsequence which converges to $M_0$ which is a contradiction. Hence, we have the desired claim that the original sequence must converge to $M_0$.
\end{proof}

\subsection{Proof of Corollary~\ref{cor-vol-thm}}

\begin{proof}
Consider $\tilde{g}_j = \frac{1}{1-\frac{1}{2j}} g_j$ and $\tilde{M}_j=(M,\tilde{g}_j)$. 
Then $ \tilde{g}_j \ge g_0$ and
\begin{align}
\vol_j(\tilde{M}_j) &= \left(1-\tfrac{1}{2j} \right)^{\frac{n}{2}} \vol_j(M_j) \rightarrow \vol(M_0)
\\ \diam_j(\tilde{M}_j) &= \left(1-\tfrac{1}{2j} \right)^{\frac{1}{2}} \diam_j(M_j) \le D_0.
\end{align}
Hence $\tilde{M}_j$ satisfies the hypotheses of Theorem \ref{vol-thm} which implies
\begin{align}
\tilde{M}_j \VFto M_0.
\end{align}
On the other hand, by construction we have  $\|g_j-\tilde{g}_j\|_{C^0_{g_0}(M)} \rightarrow 0$ which implies 
\begin{align}
\sup_{p,q\in M}|d_j(p,q)-\tilde{d}_j(p,q)|\rightarrow 0,
\end{align}
and since $\tilde{g}_j \ge g_j$ we can apply Theorem \ref{est-SWIF} with $W_j=M$ to find
\begin{align}
d_{\mathcal{F}}(\tilde{M}_j, M_j) \rightarrow 0.
\end{align}
Hence by the triangle inequality for the intrinsic flat distance we find
\begin{align}
M_j \VFto M_0.
\end{align}
\end{proof}

\subsection{Statement and Proof of Corollary~\ref{cor-Lp-thm}}

We now state a corollary which uses the observation that if one has $L^{\frac{m}{2}}$ convergence of the metric tensors $g_j$ to $g_0$ then that implies volume convergence. This perspective on Theorem \ref{vol-thm} 
says that if we have $L^p$, $p\ge \frac{m}{2}$ convergence of Riemannian manifolds then we can use Theorem \ref{vol-thm} to bootstrap up to the stronger notion of volume preserving intrinsic flat convergence by ensuring a diameter bound and a $C^0$ bound from below. 

\begin{cor} \label{cor-Lp-thm}
Suppose we have a fixed  compact oriented Riemannian manifold, $M_0=(M,g_0)$,
without boundary and
a sequence of metric tensors $g_j$ on $M$ defining $M_j=(M, g_j)$ with
\be \label{g_j-below-vol-thm}
\left(1 - \tfrac{1}{2j} \right)g_0(v,v) \le g_j(v,v) \qquad \forall v\in T_pM
\ee 
and a uniform upper bound on diameter
\be
\diam_j(M_j) \le D_0
\ee
and $L^p$ convergence with $p\ge \frac{m}{2}$
\be
\left (\int_M|g_j-g_0|_{g_0}^p dV_{g_0} \right)^{\frac{1}{p}} \rightarrow 0
\ee
then 
\be
M_j \VFto M_0.
\ee
\end{cor}

\begin{proof}
By Lemma 2.5 and Lemma 2.7 of \cite{Allen-Sormani-2} we know that $L^p$ convergence with $p \ge \frac{m}{2}$ implies
\be
\int_M |g_j|_{g_0}^{\frac{m}{2}} dV_{g_0}  \rightarrow \int_M |g_0|_{g_0}^{\frac{m}{2}} dV_{g_0} .
\ee
Then by the hypothesis that $g_j \ge \left(1 - \tfrac{1}{2j} \right) g_0$ and Lemma 4.3 of \cite{Allen-Sormani-2} we find volume convergence which allows us to use Corollary \ref{cor-vol-thm} to finish the proof.
\end{proof}

\begin{rmrk}\label{rmrk-diffeo}
It should also be noted that if one has a sequence of Riemannian manifolds which are diffeomorphic and
one can find a sequence of diffeomorphisms such that the pull back metrics satisfy the hypotheses of the
corollary stated above, then one obtains $M_j \VFto M_0$ as well.
\end{rmrk}

  \bibliographystyle{alpha}
 \bibliography{allen}

\end{document}